\RequirePackage{fix-cm}
\documentclass{svjour3}
\journalname{Research in Mathematical Sciences}

\smartqed
\usepackage{graphicx}
\usepackage{mathptmx}
\usepackage{stmaryrd}
\usepackage{amssymb}
\usepackage{tabularx,epsfig,color}
\usepackage{epstopdf}
\usepackage{pdfpages}
\usepackage{algorithm,algorithmic,multirow}
\usepackage{epsfig}
\usepackage{mathtools}
\input psfig.sty

\newcommand{\tabincell}[2]{\begin{tabular}{@{}#1@{}}#2\end{tabular}}
\newcommand{\bx}{{\bf x} }

\newcommand{\beU}{{ U} }

\newcommand{\pt}{\partial}
\newcommand{\nn}{\nonumber}
\newcommand{\eps}{\varepsilon}

\makeatletter 
\@addtoreset{equation}{section}
\@addtoreset{figure}{section}
\@addtoreset{table}{section}
\makeatother  
\renewcommand\theequation{{\thesection}%
	.{\arabic{equation}}}
\renewcommand\thefigure{{\thesection}.{\arabic{figure}}}
\renewcommand\thetable{{\thesection}.{\arabic{table}}}

\newcommand{\be}{\begin{equation}}
\newcommand{\ee}{\end{equation}}
\newcommand{\ba}{\begin{array}}
    \newcommand{\ea}{\end{array}}
\newcommand{\bea}{\begin{eqnarray}}
\newcommand{\eea}{\end{eqnarray}}
\newcommand{\beas}{\begin{eqnarray*}}
	\newcommand{\eeas}{\end{eqnarray*}}

\usepackage{setspace}
\usepackage{verbatim}
\usepackage{tikz}
\definecolor{color1}{RGB}{249,255,195}
\definecolor{color2}{RGB}{0,176,80}
\usetikzlibrary{arrows.meta, arrows}
\usetikzlibrary{shapes,snakes}
\tikzset{%
	>={Latex[width=2mm,length=2mm]},
	base/.style = {draw=black,
		minimum width=2cm, minimum height=1cm,
		text centered, fill=color1},
	Dirac/.style = {ellipse, base},
	others/.style = {base, rectangle, rounded corners},
}

\begin{document}
		
\title{A fourth-order compact time-splitting Fourier pseudospectral method for the Dirac equation\thanks{This work was partially supported by the Ministry of Education of Singapore grant
R-146-000-223-112 (MOE2015-T2-2-146).}}

\author{Weizhu Bao         \and
	Jia Yin 
}

\institute{Weizhu Bao \at
	Department of Mathematics,  National University of
	Singapore, Singapore 119076, Singapore\\
	\email{matbaowz@nus.edu.sg}          \\
	URL: http://www.math.nus.edu.sg/\~{}bao/
	\and
	Jia Yin (Corresponding author)\at
	NUS Graduate School for Integrative Sciences and Engineering (NGS), National University of
	Singapore, Singapore 117456, Singapore\\
	\email{e0005518@u.nus.edu}
}
	
	\date{Received: date / Accepted: date}
	
	\maketitle

		
\begin{abstract}
We propose a new
fourth-order compact time-splitting ($S_\text{4c}$) Fourier pseudospectral method for the Dirac equation by splitting the Dirac equation into two parts together with using the double commutator between them to integrate the Dirac equation at each time interval.
The method is explicit, fourth-order in time and spectral order in space.
It is unconditional stable and conserves the total probability in the discretized level.
It is called a compact time-splitting method since, at each time step, the number of sub-steps in $S_\text{4c}$ is much less than
those of the standard fourth-order splitting method and the fourth-order partitioned Runge-Kutta splitting method. Another advantage of
$S_\text{4c}$ is that it avoids to use negative time steps in integrating sub-problems at each time interval. Comparison between $S_\text{4c}$
and many other existing time-splitting methods for the Dirac equation are carried out in terms of accuracy and efficiency as well as long time behavior. Numerical results demonstrate the advantage in terms of
efficiency and accuracy of the proposed
$S_\text{4c}$. Finally we report the spatial/temporal resolutions of $S_\text{4c}$ for the Dirac
equation in different parameter regimes including the nonrelativistic limit regime, the semiclassical limit regime,
and the simultaneously nonrelativisic and massless limit regime.

\keywords{Dirac equation \and fourth-order compact time-splitting \and double commutator \and probability conservation \and nonrelativistic limit regime \and semiclassical limit regime}
\end{abstract}

	\section{Introduction}\setcounter{equation}{0}
	
	The Dirac equation was proposed by British physicist Paul Dirac in 1928 in order to integrate special
	relativity with quantum mechanics \cite{Dirac}. It
	successfully solved the problem that the probability density could be negative in the Klein-Gordon equation proposed by Oskar Klein and Walter Gordon in 1926 \cite{Davydov}. The Dirac equation
describes the motion of relativistic spin-1/2 massive particles,
	such as electrons and quarks. It fully explained the hydrogen
spectrum and predicted the existence of antimatter.
Recently, the Dirac equation has been extensively adopted to investigate theoretically the structures and/or dynamical
properties of graphene and graphite as well as other two-dimensional (2D) materials \cite{AMP,FW,NGMJ,NGP},
and to study the relativistic effects in molecules in super intense
lasers, e.g., attosecond lasers \cite{BCLL,FLB}.

Consider the Dirac equation with electromagnetic potentials in three spatial dimensions (3D) \cite{Dirac,Dirac2,Dirac3,Thaller}
	\begin{equation}
	\label{LDirac}
	i\hbar\partial_t\Psi =  \left(- ic\hbar\sum_{j = 1}^{3}
	\alpha_j\partial_j +mc^2\beta \right)\Psi+ e\left(V(\mathbf{x})I_4 - \sum_{j = 1}^{3}A_j(\mathbf{x})\alpha_j\right)\Psi,
	\quad \mathbf{x}\in{\mathbb{R}^3},
	\end{equation}
where $t$ is time, $\mathbf{x} = (x_1, x_2, x_3)^T$ (or $\mathbf{x} = (x, y, z)^T$) is the spatial coordinate, $\Psi:=\Psi(t, \mathbf{x}) =  (\psi_1(t, \mathbf{x})$, $\psi_2(t, \mathbf{x}),
	\psi_3(t, \mathbf{x}),\psi_4(t, \mathbf{x}))^T \in \mathbb{C}^4$ is the complex-valued spinor wave function, and $\partial_j$ represents $\partial_{x_j}$ for $j=1,2,3$. The constants used in the equation are:
$i = \sqrt{-1}$, $\hbar$ is the Planck constant, $m$ is the mass, $c$ is the speed of light and $e$ is the unit charge. In addition,
$V:=V(\mathbf{x})$ is the time-independent
	electric potential and $\mathbf{A}:=\mathbf{A}(\mathbf{x}) = (A_1(\mathbf{x}), A_2(\mathbf{x}), A_3(\mathbf{x}))^T$ stands for the
time-independent magnetic potential, which are all real-valued given functions. Finally, the $4\times 4$ matrices $\beta$ and $\alpha_j$ ($j=1,2,3)$ are the Dirac representation matrices of the four-dimensional
Clifford algebra, which are  given as
	\begin{equation}
	\label{matrices}
	\beta = \begin{pmatrix} I_2 & \mathbf{0}\\ \mathbf{0} & -I_2\end{pmatrix}, \quad
	\alpha_j = \begin{pmatrix} \mathbf{0} & \sigma_j\\ \sigma_j & \mathbf{0}\end{pmatrix},
	\quad j = 1, 2, 3,
	\end{equation}
	where $I_n$ is the $n\times n$ identity matrix and $\sigma_j$ ($j=1,2,3$) are the Pauli matrices defined as:
	\begin{equation}
	\label{Pauli}
	\sigma_1 = \begin{pmatrix} 0 & 1\\ 1 & 0 \end{pmatrix}, \quad
	\sigma_2 = \begin{pmatrix} 0 & -i \\ i & 0\end{pmatrix}, \quad
	\sigma_3 = \begin{pmatrix} 1 & 0 \\ 0 & -1 \end{pmatrix}.
	\end{equation}
	
In order to nondimensionalize the Dirac equation (\ref{LDirac}), we take
	\begin{equation}
	\label{nondim1}
	\tilde{\mathbf{x}} = \dfrac{\mathbf{x}}{x_s}, \quad \tilde{t} = \dfrac{t}{t_s}, \quad
	\tilde{V} = \dfrac{V}{A_s}, \quad \tilde{\mathbf{A}} = \dfrac{\mathbf{A}}{A_s}, \quad
	\tilde{\Psi}(\tilde{t},\tilde{\mathbf{x}} ) = \dfrac{\Psi(t, \mathbf{x})}{\psi_s},
	\end{equation}
where $x_s$, $t_s$ and $m_s$ are
length unit, time unit and mass unit,
respectively, to be taken for the nondimensionalization of the Dirac equation \eqref{LDirac}. Plugging \eqref{nondim1} into
(\ref{LDirac}) and taking $\psi_s = x_s^{-3/2}$ and
$A_s = \frac{m_sx_s^2}{et_s^2}$, after some simplification and then removing all $\tilde{}$, we obtain the dimensionless Dirac equation in 3D
	\begin{equation}
	\label{nondimr2}
	i\delta\partial_t\Psi =  \left(- i\dfrac{\delta}{\eps}\sum_{j = 1}^{3}\alpha_j\partial_j +
	\dfrac{\nu}{\varepsilon^2}\beta\right)\Psi
	+ \left(V(\mathbf{x})I_4 - \sum_{j = 1}^{3}A_j(\mathbf{x})\alpha_j\right)\Psi, \quad \mathbf{x}\in\mathbb{R}^3,
	\end{equation}
where the three dimensionless parameters $0<\eps,\delta,\nu\le1$ are given as
	\begin{equation}
	\label{nondimconst}
	\eps = \dfrac{x_s}{t_sc}=\dfrac{v_s}{c}, \quad \delta = \frac{\hbar t_s}{m_sx_s^2}, \quad \nu = \dfrac{m}{m_s},
	\end{equation}
with $v_s=x_s/t_s$ the velocity unit for nondimensionalization. In fact,
here $\eps$ represents the ratio between the
wave velocity and the speed of light, i.e. it is inversely proportional to the speed of light,  $\delta$ stands for the scaled Planck constant  and $\nu$ is the ratio between the mass of the particle and the mass unit taken
for the nondimensionalization.

  As discussed in \cite{BCJT2}, under proper assumption on the
electromagnetic potentials $V(\bx)$ and $\mathbf{A}(\bx)$,
the Dirac equation \eqref{nondimr2} in 3D could be reduced
to two dimensions (2D) and one dimension (1D). Specifically,
the Dirac equation in 2D has been widely applied to model the electron structure and dynamical properties of graphene and other 2D materials
as they share the same dispersion relation on certain points
called Dirac points \cite{FW,FW2,FGNMPC,NGP}. In fact,
the Dirac equation \eqref{nondimr2} in 3D and its dimension reduction
in 2D and 1D can be formulated in a unified way
in $d$-dimensions ($d = 1, 2, 3$) as
\begin{equation}
\label{nondimuni}
	i\delta\partial_t\Psi =  \left(- i\dfrac{\delta}{\eps}\sum_{j = 1}^{d}\alpha_j\partial_j +
	\dfrac{\nu}{\varepsilon^2}\beta\right)\Psi
	+ \left(V(\mathbf{x})I_4 - \sum_{j = 1}^{d}A_j(\mathbf{x})\alpha_j\right)\Psi, \quad \mathbf{x}\in\mathbb{R}^d,
	\end{equation}
where $\bx=(x_1,x_2)^T$ (or $\bx=(x,y)^T$) in 2D and $\bx=x_1$ (or $\bx=x$) in 1D. To study the dynamics of the Dirac equation (\ref{nondimuni}),
the initial condition is usually taken as
	\begin{equation}
	\label{initial}
	\Psi(t = 0, \mathbf{x}) = \Psi_0(\mathbf{x}), \quad \mathbf{x}\in\mathbb{R}^d.
	\end{equation}

The Dirac equation (\ref{nondimuni}) with \eqref{initial} is dispersive, time-symmetric,
and it conserves the total {\sl probability} \cite{BCJT2}
    \begin{equation}
    \label{mass}
    \|\Psi(t, \cdot)\|^2 := \int_{\mathbb{R}^d}|\Psi(t, \mathbf{x})|^2d\mathbf{x} =
    \int_{\mathbb{R}^d}\sum_{j = 1}^4|\psi_j(t, \mathbf{x})|^2d\mathbf{x}
    \equiv\|\Psi(0, \cdot)\|^2 = \|\Psi_0\|^2, \ \ t \geq 0,
    \end{equation}
and the {\sl energy} \cite{BCJT2}
\begin{eqnarray}
\label{energy}
	E(\Psi(t,\cdot)) &:=& \int_{\mathbb{R}^d}\left(- i\dfrac{\delta}{\eps}\sum_{j = 1}^{d}\Psi^*\alpha_j\partial_j\Psi +
	\dfrac{\nu}{\varepsilon^2}\Psi^*\beta\Psi
	+ V(\mathbf{x})|\Psi|^2 - \sum_{j = 1}^{d}A_j(\mathbf{x})\Psi^*\alpha_j\Psi\right) d\mathbf{x}\nonumber\\
	&\equiv& E(\Psi_0), \qquad t\ge0,
	\end{eqnarray}
where $\Psi^* = \overline{\Psi}^T$ with $\overline{f}$ denoting the complex conjugate of $f$.

Introduce the total probability density$\rho:=\rho(t,\bx)$ as
\begin{equation}
    \label{density}
    \rho(t, \mathbf{x}) = \sum_{j = 1}^4\rho_j(t, \mathbf{x}) = \Psi(t,\bx)^*\Psi(t,\bx),\quad \mathbf{x}\in\mathbb{R}^d,
    \end{equation}
where the probability density $\rho_j:=\rho_j(t,\bx)$ of the $j$-th ($j=1,2,3,4$) component is defined as
    \begin{equation}
    \rho_j(t, \mathbf{x}) = |\psi_j(t, \mathbf{x})|^2,\quad \mathbf{x}\in\mathbb{R}^d,
    \end{equation}
and the current density $\mathbf{J}(t, \mathbf{x}) = (J_1(t, \mathbf{x}),
    \ldots, J_d(t, \mathbf{x})))^T$  as
    \begin{equation}
    \label{current}
    J_l(t, \mathbf{x}) = \frac{1}{\eps}\Psi(t,\bx)^*\alpha_l\Psi(t,\bx), \quad l = 1,\ldots,d,
    \end{equation}
then the following conservation law can be obtained from the
 Dirac equation (\ref{nondimuni}) \cite{BCJT2}
 \begin{equation}
 \label{cons11}
    \partial_t\rho(t, \mathbf{x}) + \nabla\cdot\mathbf{J}(t, \mathbf{x}) = 0,
    \quad \mathbf{x}\in\mathbb{R}^d, \quad t\geq 0.
    \end{equation}

If the electric potential $V$ is perturbed by a real constant $V^0$,
i.e., $V\to V + V^0$, then the solution
$\Psi(t,\bx)\to e^{-i\frac{V^0t}{\delta}}\Psi(t, \mathbf{x})$,
which implies that the probability density of each component
$\rho_j (j = 1, 2, 3, 4)$ and the total probability density $\rho$ are
all unchanged. In addition, when $d = 1$,
if the magnetic potential $A_1$ is perturbed by a real constant $A_1^0$,
i.e., $A_1\to A_1 + A_1^0$, then the solution
$\Psi(t,\bx)\to e^{i\frac{A_1^0t}{\delta}\alpha_1}\Psi(t, \mathbf{x})$, which implies that only the total probability density $\rho$ is unchanged; however,
this property is unfortunately not valid in 2D and 3D.
Furthermore, if the external electromagnetic potentials are all real constants, i.e. $V(\mathbf{x}) \equiv V^0$ and $A_j(\mathbf{x}) \equiv A_j^0$ ($j = 1, \ldots,d$) with $\mathbf{A}^0 = (A_1^0, ...,
	A_d^0)^T$, the Dirac equation (\ref{nondimuni}) admits the plane wave solution
	$\Psi(t, \mathbf{x}) = \mathbf{B}e^{i(\mathbf{k}\cdot\mathbf{x} - \frac{\omega}{\delta} t)}$ with $\omega$  the time
	frequency, $\mathbf{B}\in\mathbb{R}^4$  the amplitude vector and
	$\mathbf{k} = (k_1, ..., k_d)^T\in \mathbb{R}^d$  the spatial wave number, which satisfies the following {eigenvalue problem}
	\begin{equation}
	 \omega\mathbf{B}=\left(\sum_{j = 1}^d\left(\frac{\delta k_j}{\eps} - A_j^0\right)\alpha_j + \frac{\nu}{\eps^2}\beta
		+ V^0I_4\right)\mathbf{B}.
	\end{equation}
Solving the above equation, we can get
the {\sl dispersion relation} of the Dirac equation (\ref{nondimuni})
	\begin{equation}
\label{disp}
	\omega := \omega(\mathbf{k}) = V^0 \pm \frac{1}{\eps^2}\sqrt{\nu^2 + \eps^2|\delta\mathbf{k} -
		\eps\mathbf{A}^0|^2}, \quad \mathbf{k}\in\mathbb{R}^d.
	\end{equation}
	
In 2D and 1D, i.e. $d=2$ or $1$ in (\ref{nondimuni}),
similar as those in \cite{BCJT}, the Dirac equation (\ref{nondimuni})
can be decoupled into two simplified PDEs with
$\Phi:=\Phi(t,\bx)=(\phi_1(t,\bx),\phi_2(t,\bx))^T\in\mathbb{C}^2$ satisfying
\begin{equation}
	\label{LDirac1d2d}
	i\delta\partial_t\Phi = \left(-i\dfrac{\delta}{\eps}\sum_{j = 1}^{d}\sigma_j\partial_j + \dfrac{\nu}{\eps^2}
	\sigma_3\right)\Phi + \left(V(\mathbf{x})I_2 - \sum_{j = 1}^{d}A_j(\mathbf{x})\sigma_j\right)\Phi, \quad \mathbf{x} \in\mathbb{R}^d,
	\end{equation}
where $\Phi = (\psi_1, \psi_4)^T$ (or $\Phi=(\psi_2,\psi_3)^T$).
Again, to study the dynamics of the Dirac equation (\ref{LDirac1d2d}),
the initial condition is usually taken as
	\begin{equation}
	\label{initial11}
	\Phi(t = 0, \mathbf{x}) = \Phi_0(\mathbf{x}), \quad \mathbf{x}\in\mathbb{R}^d.
	\end{equation}
Similarly, the Dirac equation (\ref{LDirac1d2d}) with \eqref{initial11} is dispersive, time-symmetric,
and it conserves the total {\sl probability} \cite{BCJT2}
   \bea
   \label{mass11}
    \|\Phi(t, \cdot)\|^2 &:=& \int_{\mathbb{R}^d}|\Phi(t, \mathbf{x})|^2d\mathbf{x} =
    \int_{\mathbb{R}^d}\sum_{j = 1}^2|\phi_j(t, \mathbf{x})|^2d\mathbf{x}\nonumber\\
    &\equiv& \|\Phi(0, \cdot)\|^2 = \|\Phi_0\|^2, \quad t \geq 0,
    \eea
and the {\sl energy} \cite{BCJT2}
\begin{eqnarray}
\label{energy11}
	E(\Phi(t,\cdot)) &:=& \int_{\mathbb{R}^d}\left(- i\dfrac{\delta}{\eps}\sum_{j = 1}^{d}\Phi^*\sigma_j\partial_j\Phi +
	\dfrac{\nu}{\varepsilon^2}\Phi^*\sigma_3\Phi
	+ V(\mathbf{x})|\Phi|^2 - \sum_{j = 1}^{d}A_j(\mathbf{x})\Phi^*\sigma_j\Phi\right) d\mathbf{x}\nonumber\\
	& \equiv& E(\Phi_0), \qquad t\ge0.
	\end{eqnarray}
Again, introduce the total probability density $\rho:=\rho(t,\bx)$ as
\begin{equation}
    \label{density11}
    \rho(t, \mathbf{x}) = \sum_{j = 1}^2\rho_j(t, \mathbf{x}) = \Phi(t,\bx)^*\Phi(t,\bx),\quad \mathbf{x}\in\mathbb{R}^d,
    \end{equation}
where the probability density $\rho_j:=\rho_j(t,\bx)$ of the $j$-th ($j=1,2$) component is defined as
    \begin{equation}
    \rho_j(t, \mathbf{x}) = |\phi_j(t, \mathbf{x})|^2,\quad \mathbf{x}\in\mathbb{R}^d,
    \end{equation}
and the current density $\mathbf{J}(t, \mathbf{x}) = (J_1(t, \mathbf{x}),
    \ldots, J_d(t, \mathbf{x})))^T$  as
    \begin{equation}
    \label{current11}
    J_l(t, \mathbf{x}) = \frac{1}{\eps}\Phi(t,\bx)^*\sigma_l\Phi(t,\bx), \quad l = 1,\ldots,d,
    \end{equation}
then the same conservation law \eqref{cons11} can be obtained from the
 Dirac equation (\ref{LDirac1d2d}) \cite{BCJT2}.

Similarly, if the electric potential $V$ is perturbed by a real constant $V^0$,
i.e., $V\to V + V^0$, then the solution
$\Phi(t,\bx)\to e^{-i\frac{V^0t}{\delta}}\Phi(t, \mathbf{x})$,
which implies that the probability density of each component
$\rho_j (j = 1, 2)$ and the total probability density $\rho$ are
all unchanged. In addition, when $d = 1$,
if the magnetic potential $A_1$ is perturbed by a real constant $A_1^0$,
i.e., $A_1\to A_1 + A_1^0$, then the solution
$\Phi(t,\bx)\to e^{i\frac{A_1^0t}{\delta}\sigma_1}\Phi(t, \mathbf{x})$, which implies that only the total probability density $\rho$ is unchanged; however,
this property is unfortunately not valid in 2D.
Furthermore, if the external electromagnetic potentials are all real constants, i.e. $V(\mathbf{x}) \equiv V^0$ and $A_j(\mathbf{x}) \equiv A_j^0$ ($j = 1, \ldots,d$) with $\mathbf{A}^0 = (A_1^0, ...,
	A_d^0)^T$, the Dirac equation (\ref{LDirac1d2d}) admits the plane wave solution
	$\Phi(t, \mathbf{x}) = \mathbf{B}e^{i(\mathbf{k}\cdot\mathbf{x} - \frac{\omega}{\delta} t)}$ with $\omega$  the time
	frequency, $\mathbf{B}\in\mathbb{R}^2$  the amplitude vector and
	$\mathbf{k} = (k_1, ..., k_d)^T\in \mathbb{R}^d$  the spatial wave number, which satisfies the following {eigenvalue problem}
	\begin{equation}
	 \omega\mathbf{B}=\left(\sum_{j = 1}^d\left(\frac{\delta k_j}{\eps} - A_j^0\right)\sigma_j + \frac{\nu}{\eps^2}\sigma_3
		+ V^0I_2\right)\mathbf{B}.
	\end{equation}
Solving the above equation, we can get
the {\sl dispersion relation} of the Dirac equation (\ref{LDirac1d2d})
	\begin{equation}
\label{disp11}
	\omega := \omega(\mathbf{k}) = V^0 \pm \frac{1}{\eps^2}\sqrt{\nu^2 + \eps^2|\delta\mathbf{k} -
		\eps\mathbf{A}^0|^2}, \quad \mathbf{k}\in\mathbb{R}^d.
	\end{equation}

If one sets the mass unit $m_s=m$, length unit $x_s =\frac{\hbar}{mc}$,
 and time unit $t_s=\frac{x_s}{c}=\frac{\hbar}{mc^2}$,
then $\eps=\delta=\nu=1$, which corresponds
to the classical (or standard) scaling. This choice of $x_s$, $m_s$ and $t_s$ is appropriate when the wave speed is at the same order of the speed of light. However, a different choice of $x_s$, $m_s$ and $t_s$ is more appropriate when the wave speed is much smaller than the speed of light. We remark
here that the choice of $x_s$, $m_s$ and $t_s$ determines the observation scale of time evolution of the
system and decides which phenomena can be resolved by discretization on specified spatial/temporal grids and which phenomena is ‘visible’ by asymptotic analysis.

Different parameter regimes could be considered for the Dirac equation
\eqref{nondimuni} (or \eqref{LDirac1d2d}), which are
displayed in Fig. \ref{scaling_summary}:

	
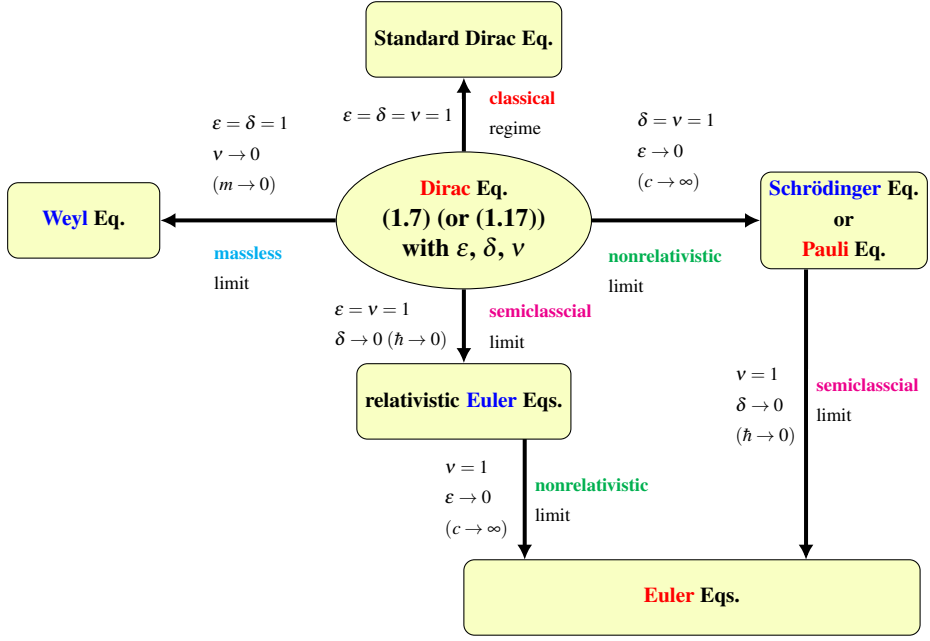
\begin{figure}
	\begin{tikzpicture}
		[node distance=10mm,
		every node/.style={fill=white}, align=center]
		\node(Dirac)    [Dirac, align = center]  {\footnotesize \textbf{\textcolor{red}{Dirac} Eq.}\\ \textbf{(\ref{nondimuni}) (or (\ref{LDirac1d2d}))} \\ \textbf{with $\eps$, $\delta$, $\nu$}};
		\node(standard)  [others, above of  = Dirac, yshift = 1.4cm]  {\footnotesize \textbf{Standard Dirac Eq.}};
		\node(Weyl)     [others, left of = Dirac, xshift=-40mm]   {\footnotesize \textbf{\textcolor{blue}{Weyl} Eq.}};
		\node(Schrodinger)  [others, right of = Dirac, xshift = 40mm, align = center]
		{\textbf{\footnotesize\textcolor{blue}{Schr\"{o}dinger} Eq.} \\ \footnotesize\textbf{or}
			\\ \footnotesize\textbf{\textcolor{red}{Pauli} Eq.}};
		\node(rEuler)   [others, below of = Dirac, yshift = -14mm]  {\footnotesize\textbf{relativistic \textcolor{blue}{Euler}
				Eqs.}};
		\node(Euler)    [others, below of = rEuler, xshift = 3cm, yshift = -1.6cm, minimum width = 6cm]
		{\footnotesize\textbf{\textcolor{red}{Euler} Eqs.}};
		
		
		\draw [ ->,line width=1.5pt]       (Dirac) -- node[above, align = flush left, xshift = 0.32mm, yshift = 2mm]{\scriptsize
			$\varepsilon = \delta = 1$\\ \scriptsize $\nu\rightarrow 0$\\ \scriptsize$(m\rightarrow 0)$}
		node[below, align = flush left, yshift = -2mm]{\scriptsize{\textbf{\textcolor{cyan}{massless}}}\\
			\scriptsize{limit}}(Weyl);
		\draw [ ->,line width=1.5pt]       (Dirac) -- node[above, align = flush left, yshift =
		2.4mm]{\scriptsize $\delta = \nu = 1$ \\ \scriptsize $\varepsilon\rightarrow 0$\\
			\scriptsize $(c\rightarrow\infty)$}
		node[below, align = flush left, xshift = -1.4mm, yshift = -2.4mm]{\scriptsize \textbf{\textcolor{color2}{nonrelativistic}}
			\\ \scriptsize limit}(Schrodinger);
		\draw [ ->,line width=1.5pt]       (Dirac) -- node[left]{\scriptsize$\varepsilon = \delta = \nu = 1$}
		node[right, align = flush left, xshift = 2mm]{\scriptsize\textbf{\textcolor{red}{classical}} \\
			\scriptsize{regime}} (standard);
		\draw [ ->,line width=1.5pt]     (Dirac) -- node[left, align = flush left, xshift = -0.8mm]
		{\scriptsize $\varepsilon = \nu = 1$ \\
			\scriptsize $\delta\rightarrow 0$ $(\hbar\rightarrow 0)$}
		node[right, align = flush left, xshift = 2mm]{\scriptsize \textbf{\textcolor{magenta}{semiclasscial}} \\
			\scriptsize {limit}} (rEuler);
		\draw [ ->,line width=1.5pt]   (rEuler)+(8mm, -5mm) --
		node[left, align = flush left, xshift = -0.8mm]{\scriptsize $\nu = 1$
			\\\scriptsize $\varepsilon\rightarrow 0$ \\ \scriptsize $(c\rightarrow \infty)$}
		node[right, align = flush left]{\scriptsize \textbf{\textcolor{color2}{nonrelativistic}} \\ \scriptsize {limit}}
		++(8mm, -21mm)(Euler);
		\draw [ ->,line width=1.5pt, in = -60]    (Schrodinger)+(-5mm, -6.5mm) --
		node[left, align = flush left, yshift = 1mm]
		{\scriptsize $\nu = 1$ \\ \scriptsize $\delta\rightarrow 0$ \\ \scriptsize $(\hbar\rightarrow 0)$}
		node[right, align = flush left, yshift = 2mm]{\scriptsize\textbf{\textcolor{magenta}{semiclasscial}}　\\
			\scriptsize {limit}}
		++(-5mm, -45mm)(Euler);
	\end{tikzpicture}
	\caption{Diagram of different parameter regimes and limits of the Dirac equation (\ref{nondimuni}) (or \eqref{LDirac1d2d}).}
	\label{scaling_summary}
\end{figure}

\begin{itemize}
\item Standard (or classical) regime, i.e. $\eps =\delta =\nu = 1$ ($\Longleftrightarrow m_s=m$,
$x_s =\frac{\hbar}{mc}$,
 and $t_s=\frac{\hbar}{mc^2}$), the wave speed is at
the order of the speed of light. In this parameter regime,
formally the dispersion relation \eqref{disp} (or \eqref{disp11}) suggests
$\omega(\mathbf{k})=O(1)$ when $|\mathbf{k}|=O(1)$ and thus
the solution propagates waves with wavelength at $O(1)$ in space and time.
In addition, if the initial data $\Psi_0=O(1)$ in \eqref{initial} (or $\Phi_0=O(1)$ in \eqref{initial11}), then the solution $\Psi=O(1)$ of \eqref{nondimuni} with \eqref{initial}
(or $\Phi=O(1)$ of \eqref{LDirac1d2d} with \eqref{initial11}),
which implies that the probability density $\rho=O(1)$ in \eqref{density} (or \eqref{density11}), current density
$\mathbf{J}=O(1)$ in \eqref{current} (or \eqref{current11}) and the energy $E(\Psi(t,\cdot))=O(1)$ in \eqref{energy}
(or  $E(\Phi(t,\cdot))=O(1)$ in \eqref{energy11}).
There were extensive analytical and numerical studies for the Dirac equation \eqref{nondimuni} (or \eqref{LDirac1d2d}) with
$\eps=\delta=\nu=1$ in the literatures. For the existence and multiplicity of
bound states and/or standing wave solutions, we
refer to \cite{Das,DK,ES,GGT,Gross,Ring} and references therein. In this parameter regime,
for the
numerical part, many efficient and accurate numerical methods have been proposed and analyzed \cite{AL}, such as the finite difference time domain (FDTD)
methods \cite{ALSFB,NSG}, time-splitting Fourier pseudospectral (TSFP) method \cite{BCJT2,HJM,Cao}, exponential wave integrator Fourier pseudospectral (EWI-FP) method \cite{BCJT2}, the Gaussian
beam method \cite{WHJY}, etc.

\item Massless limit regime, i.e. $\eps = \delta = 1$ and $0 < \nu \ll 1$ ($\Longleftrightarrow x_s =\frac{\hbar}{m_sc}$ and $t_s=\frac{\hbar}{m_sc^2}$), the mass of the particle is much less than
    the mass unit. In this parameter regime, the Dirac equation \eqref{nondimuni} (or \eqref{LDirac1d2d}) converges -- regularly -- to the Weyl equation \cite{Ohlsson,XBAN} with linear convergence rate in terms of $\nu$. Any numerical methods for
the Dirac equation \eqref{nondimuni} (or \eqref{LDirac1d2d}) in the standard regime can be applied in this parameter regime.

\item Nonrelativistic limit regime, i.e. $\delta = \nu = 1$ and $0 < \eps\ll 1$ ($\Longleftrightarrow m_s=m$ and
$t_s =\frac{mx_s^2}{\hbar}$), i.e. the wave speed is much less than the speed of light.
In this parameter regime, formally the dispersion relation \eqref{disp} (or \eqref{disp11}) suggests
$\omega(\mathbf{k})=\eps^{-2}+O(1)$ when $|\mathbf{k}|=O(1)$ and thus
the solution propagates waves with wavelength at
$O(\eps^2)$ and $O(1)$ in time and space, respectively, when $0<\eps\ll 1$.
In addition, if the initial data $\Psi_0=O(1)$ in \eqref{initial} (or $\Phi_0=O(1)$ in \eqref{initial11}), then the solution $\Psi=O(1)$ of \eqref{nondimuni} with \eqref{initial}
(or $\Phi=O(1)$ of \eqref{LDirac1d2d} with \eqref{initial11}),
which implies that the probability density $\rho=O(1)$ in \eqref{density} (or \eqref{density11}), current density
$\mathbf{J}=O(\eps^{-1})$ in \eqref{current} (or \eqref{current11}) and the energy $E(\Psi(t,\cdot))=O(\eps^{-2})$ in \eqref{energy}
(or  $E(\Phi(t,\cdot))=O(\eps^{-2})$ in \eqref{energy11}).
The highly oscillatory nature of the solution in time
and the unboundedness of the energy bring significant difficulty
in mathematical analysis and numerical simulation of the Dirac equation
in the nonrelativistic regime, i.e. $0<\eps\ll1$.
By diagonalizing the Dirac operator and using proper ansatz,
one can show that the Dirac equation \eqref{nondimuni} (or \eqref{LDirac1d2d}) converges -- singularly -- to the Pauli equation \cite{BMP,Hunziker}
 and/or the Schr\"{o}dinger equation \cite{Bad,BMP} when $\eps\to0^+$.
Rigorous error estimates were established for the FDTD, TSFP and EWI-FP
methods in this parameter regime \cite{BCJT2}, which depend explicitly on
the mesh size $h$, time step $\tau$ and the small parameter $\eps$.
Recently, a uniformly accurate multiscale time integrator pseudospectral method was proposed and analyzed for the Dirac equation
in the nonrelativistic limit regime, which converges uniformly
with respect to $\eps\in(0,1]$ \cite{BCJT,LMZ}.	
		
\item Semiclassical limit regime, i.e. $\eps = \nu = 1$
and  $0 < \delta \ll 1$ ($\Longleftrightarrow m_s=m$ and $t_s=\frac{x_s}{c}$),
the quantum effect could be neglected.
In this parameter regime, the solution propagates waves
with wavelength at $O(\delta)$ in space and time \cite{BK} when
$0<\delta\ll1$. In addition, if the initial data $\Psi_0=O(1)$ in \eqref{initial} (or $\Phi_0=O(1)$ in \eqref{initial11}), then the solution $\Psi=O(1)$ of \eqref{nondimuni} with \eqref{initial}
(or $\Phi=O(1)$ of \eqref{LDirac1d2d} with \eqref{initial11}),
which implies that the probability density $\rho=O(1)$ in \eqref{density} (or \eqref{density11}), current density
$\mathbf{J}=O(1)$ in \eqref{current} (or \eqref{current11}) and the energy $E(\Psi(t,\cdot))=O(1)$ in \eqref{energy}
(or  $E(\Phi(t,\cdot))=O(1)$ in \eqref{energy11}).
The highly oscillatory nature of the solution in time and space
brings significant difficulty
in mathematical analysis and numerical simulation of the Dirac equation
in the semiclassical limit regime, i.e. $0<\delta\ll1$. By using
the Wigner transformation method, one can show that the Dirac equation \eqref{nondimuni} (or \eqref{LDirac1d2d}) converges -- singularly -- to the relativistic Euler equations \cite{AS,GM,Spohn}. Similar to the analysis of different
numerical methods for the Schr\"{o}dinger equation in the semiclassical
limit regime \cite{ABB,BC,BJP1,BJP2,Car1,Car2,JMS}, it is an interesting question to
establish rigorous error bounds of different numerical methods
for the Dirac equation in the semiclassical limit regime
such that they depend explicitly on mesh size $h$, time step $\tau$
as well as the small parameter $\delta\in (0,1]$.

\item Simultaneously nonrelativistic and massless limit regimes,
i.e. $\delta = 1$, $\nu\sim \eps$ and $0 < \eps \ll 1$ ($\Longleftrightarrow t_s =\frac{m_sx_s^2}{\hbar}$), the wave speed is much less than the speed of light and the mass of the particle is much less than
the mass unit. Here we assume $\nu=\nu_0 \eps$ with $\nu_0>0$ a
constant independent of $\eps\in(0,1]$. In this case, the Dirac equation \eqref{nondimuni} can be re-written as ($d=1,2,3$)
\begin{equation}
\label{nondimuni33}
	i\partial_t\Psi =  \left(- i\dfrac{1}{\eps}\sum_{j = 1}^{d}\alpha_j\partial_j +
	\dfrac{\nu_0}{\varepsilon}\beta\right)\Psi
	+ \left(V(\mathbf{x})I_4 - \sum_{j = 1}^{d}A_j(\mathbf{x})\alpha_j\right)\Psi, \quad \mathbf{x}\in\mathbb{R}^d,
	\end{equation}
and respectively, the Dirac equation \eqref{LDirac1d2d}) can be re-written as ($d=1,2$)
\begin{equation}
	\label{LDirac1d2d33}
	i\partial_t\Phi = \left(-i\dfrac{1}{\eps}\sum_{j = 1}^{d}\sigma_j\partial_j + \dfrac{\nu_0}{\eps}
	\sigma_3\right)\Phi + \left(V(\mathbf{x})I_2 - \sum_{j = 1}^{d}A_j(\mathbf{x})\sigma_j\right)\Phi, \quad \mathbf{x} \in\mathbb{R}^d.
	\end{equation}
In this parameter regime, formally the dispersion relation \eqref{disp} (or \eqref{disp11}) suggests
$\omega(\mathbf{k})=O(\eps^{-1})$ when $|\mathbf{k}|=O(1)$ and thus
the solution propagates waves with wavelength at
$O(\eps)$ and $O(1)$ in time and space, respectively, when $0<\eps\ll 1$.
In addition, if the initial data $\Psi_0=O(1)$ in \eqref{initial} (or $\Phi_0=O(1)$ in \eqref{initial11}), then the solution $\Psi=O(1)$ of \eqref{nondimuni33} with \eqref{initial}
(or $\Phi=O(1)$ of \eqref{LDirac1d2d33} with \eqref{initial11}),
which implies that the probability density $\rho=O(1)$ in \eqref{density} (or \eqref{density11}), current density
$\mathbf{J}=O(\eps^{-1})$ in \eqref{current} (or \eqref{current11}) and the energy $E(\Psi(t,\cdot))=O(\eps^{-1})$ in \eqref{energy}
(or  $E(\Phi(t,\cdot))=O(\eps^{-1})$ in \eqref{energy11}).
Again, the highly oscillatory nature of the solution in time
and the unboundedness of the energy bring significant difficulty
in mathematical analysis and numerical simulation of the Dirac equation
in this parameter regime. In fact, it is an interesting question to
study the singular limit of the Dirac equation
\eqref{nondimuni33} (or  \eqref{LDirac1d2d33}) when $\eps\to 0^+$
and  establish rigorous error bounds of different numerical methods
for the Dirac equation in this parameter regime
such that they depend explicitly on mesh size $h$, time step $\tau$
as well as the small parameter $\eps\in(0,1]$.
\end{itemize}

First-order and second-order (in time) time-splitting spectral methods
have been proposed and analyzed for the Dirac equation \eqref{nondimuni} (or  \eqref{LDirac1d2d}) \cite{BCJT2}. Extension to higher order, e.g. fourth-order, time-splitting spectral methods can be done straightforward
by adapting the high order splitting methods \cite{BS,MQ,Suzuki2}, e.g.
the standard fourth-order splitting ($S_4$) \cite{FR,Suzuki0,Yoshida} or the fourth-order
partitioned Runge-Kutta ($S_\text{4RK}$) splitting method \cite{BM,Geng}. As it was observed in the literature \cite{MQ}, the $S_4$ splitting method has to use negative time step in at least one of the sub-problems at each
time interval \cite{FR,Suzuki0,Yoshida},  which causes some kind of drawbacks in practical computation,
and the number of sub-problems in the $S_\text{4RK}$ splitting method at each time interval  is much bigger than that of the $S_4$ splitting method \cite{BM},
which increases the computational cost at each time step a lot.
Motivated by the fourth-order gradient symplectic integrator
for the Sch\"{o}dinger equation invented by \cite{Chin,CC,CC2}, a new fourth-order compact time-splitting ($S_\text{4c}$) Fourier pseudospectral method will be proposed for the Dirac equation by splitting the Dirac equation into two parts together with using the double commutator between them to integrate the Dirac equation at each time interval.
The method is explicit, fourth-order in time and spectral order in space.
We compare the accuracy and efficiency as well as long time behavior of the $S_\text{4c}$ method with many other existing time-splitting methods for the Dirac equation. Numerical results demonstrate the advantage of the proposed
$S_\text{4c}$ in terms of efficiency and accuracy, especially in 1D and
high dimensions (2D and 3D) without magnetic potential. We also report the spatial/temporal resolution of the $S_\text{4c}$ method for the Dirac
equation in different parameter regimes.

The rest of the paper is organized as follows.
In section 2, we review different time-splitting schemes for
differential equations. In section 3, we calculate the double commutator
between the two parts decoupled from the Dirac equation.
A fourth-order compact time-splitting Fourier pseudospectral
method is proposed for the Dirac equation in section 4. In section 5,
we compare accuracy and efficiency as well as long time behavior
of different time-splitting methods for the Dirac equation.
In section 6, we report spatial/temporal resolution of the
fourth-order compact time-splitting Fourier pseudospectral
method for the Dirac equation in different parameter regimes.
Finally, some concluding remarks are drawn in section 7.
Throughout the paper, we adopt the
standard Sobolev spaces  and the corresponding norms and adopt $A\lesssim B$ to mean that
there exists a generic constant $C>0$ independent of $\eps$, $\tau$, $h$,
$\delta$ and  $\nu$  such that $|A|\le C\,B$.


\section{Review of different time-splitting schemes}\label{section2}
Splitting (or split-step or time-splitting) methods
have been widely used in numerically
integrating differential equations \cite{MQ}.
Combined with different spatial discretization schemes,
they have also been applied in solving
partial differential equations \cite{MQ}.
For details, we refer to \cite{Suzuki,Suzuki2,Suzuki3} and references therein.

For simplicity of notations and the convenience of readers,
here we review several time-splitting schemes
for integrating a differential equation in the form
  \begin{equation}
  \label{partial}
  \pt_tu(t, \bx) = (T + W)u(t, \bx),
  \end{equation}
with the initial data
\begin{equation}
\label{it55}
u(0, \bx) = u_0(\bx),
\end{equation}
where $T$ and $W$ are two time-independent operators.
For any time step $\tau>0$,
formally the solution of \eqref{partial} with \eqref{it55}
can be represented as
  \begin{equation}
  \label{exact}
  u(\tau, \bx) = e^{\tau(T + W)}u_0(\bx).
  \end{equation}
A splitting (or split-step or time-splitting) scheme
can be designed by approximating the operator $e^{\tau(T + W)}$
by a product of a sequence of $e^{\tau T}$ and $e^{\tau W}$ \cite{Suzuki0,Yoshida}, i.e.
  \begin{equation}
  \label{split66}
  e^{\tau(T + W)} \approx \Pi_{j = 1}^n e^{a_j\tau\, T}e^{b_j\tau\, W},
  \end{equation}
where $n\ge1$, $a_j\in{\mathbb R}$ and $b_j\in {\mathbb R}$ ($j=1,\ldots,n$) are to be
determined such that the approximation has certain order of accuracy in terms
of $\tau$ \cite{Suzuki0,Yoshida}. Without loss of generality, here we suppose that the computation for $e^{\tau W}$ is easier and/or more efficient
than that for $e^{\tau T}$.

  \subsection{First-order and second-order time-splitting methods}
Taking $n=1$ and $a_1=b_1=1$ in \eqref{split66}, one can obtain
the {\bf first-order Lie-Trotter splitting} ($S_1$) method as \cite{Trotter}
  \begin{equation}
  u(\tau, \bx) \approx S_1(\tau)u_0(\bx) := e^{\tau T}e^{\tau W}u_0(\bx).
  \end{equation}
In this method, one needs to integrate the operator $T$ once and
the operator $W$ once.
By using Taylor expansion, one can formally show the local truncation error as \cite{Strang}
\begin{equation}
\label{err_S1}
\|u(\tau,\bx)-S_1(\tau)u_0(\bx)\|\le C_1 \tau^2,
\end{equation}
where $C_1>0$ is a constant independent of $\tau$ and $\|\cdot \|$ is a norm depending on the problem. Thus the method is formally a first-order integrator \cite{MQ}.

Similarly, taking $n=2$, $a_1=0$, $b_1=\frac{1}{2}$, $a_2=1$ and
$b_2=\frac{1}{2}$, one can obtain the
{\bf second-order Strang splitting} ($S_2$) method as \cite{Strang}
  \begin{equation}
  u(\tau, \bx) \approx S_2(\tau)u_0(\bx) := e^{\frac{\tau}{2}W}e^{\tau T}e^{\frac{\tau}{2}W}u_0(\bx).
  \end{equation}
In this method, one needs to integrate the operator $T$ once and
the operator $W$ twice.
Again, by using Taylor expansion, one can formally show the local truncation error as \cite{Strang}
\begin{equation}
\label{err_S2}
\|u(\tau,\bx)-S_2(\tau)u_0(\bx)\|\le C_2 \tau^3,
\end{equation}
where $C_2>0$ is a constant independent of $\tau$ .
Thus it is formally a second-order integrator \cite{MQ}.

  \subsection{Fourth-order time-splitting methods}
  High order, especially fourth-order, splitting methods
for \eqref{partial} with \eqref{it55} via the construction
\eqref{split66} had been extensively studied in the literature \cite{Chin,CC}.

For simplicity, here we only mention a popular
{\bf fourth-order Forest-Ruth (or Yoshida) splitting}  ($S_4$) method
 \cite{FR,Suzuki0,Yoshida} as
  \begin{equation}
  u(\tau, \bx) \approx S_4(\tau)u_0(\bx) := S_2(w_1\tau)S_2(w_2\tau)S_2(w_1\tau)u_0(\bx),
  \end{equation}
where
  \begin{equation}
  w_1 = \dfrac{1}{2 - 2^{1 / 3}},
  \quad w_2 = -\dfrac{2^{1 / 3}}{2 - 2^{1 / 3}}.
  \end{equation}
In this method, one needs to integrate the operator $T$ three times and
the operator $W$ four times.
Again, by using Taylor expansion, one can formally show the local truncation error as \cite{FR}
\begin{equation}
\label{err_S4}
\|u(\tau,\bx)-S_4(\tau)u_0(\bx)\|\le C_4 \tau^5,
\end{equation}
where $C_4>0$ is a constant independent of $\tau$.
Thus it is formally a fourth-order integrator \cite{MQ}.
Due to that negative time steps, e.g. $w_2<0$, are used
in the method, in general, it cannot be applied to solve dissipative differential equations. In addition,
as it was noticed in the literature \cite{MQ},
some drawbacks of the $S_4$ method were reported,
such as the constant $C_4$ is usually much larger than
$C_1$ and $C_2$, and the fourth-order accuracy could be observed
only when $\tau$ is very small \cite{MQ,Suzuki3}.

  To overcome the drawbacks of the $S_4$ method,
the {\bf fourth-order partitioned Runge-Kutta splitting } ($S_\text{4RK}$) method
was proposed  \cite{BM,Geng} as
  \bea
  \label{4RK}
  u(\tau, \bx) &\approx& S_\text{4RK}(\tau)u_0(\bx)\\
  &:=& e^{a_1\tau W}e^{b_1\tau T}e^{a_2\tau W}
  e^{b_2\tau T}e^{a_3\tau W}e^{b_3\tau T}e^{a_4\tau W}e^{b_3\tau T}e^{a_3\tau W}e^{b_2\tau T}e^{a_2\tau W}
  e^{b_1\tau T}e^{a_1\tau W}u_0(\bx),\nonumber
  \eea
where
  \begin{equation*}
  \begin{aligned}
 & a_1 = 0.0792036964311957, \quad a_2 = 0.353172906049774,\\
 & a_3 = -0.0420650803577195, \quad a_4 = 1 - 2(a_1 + a_2 + a_3),\\
 & b_1 = 0.209515106613362, \quad b_2 = -0.143851773179818, \quad b_3 = \frac{1}{2} - (b_1 + b_2).
\end{aligned}
  \end{equation*}
In this method, one needs to integrate the operator $T$ six times and
the operator $W$ seven times.
Again, by using Taylor expansion, one can formally show the local truncation error as \cite{BM}
\begin{equation}
\label{err_4RK}
\|u(\tau,\bx)-S_\text{4RK}(\tau)u_0(\bx)\|\le \widetilde{C}_4 \tau^5,
\end{equation}
where $\widetilde{C}_4>0$ is a constant independent of $\tau$.
Thus it is formally a fourth-order integrator \cite{MQ}.
It is easy to see that the computational cost of the $S_\text{4RK}$ method
is about two times that of the $S_4$ method. In this method, negative time steps, e.g. $a_3<0$, have also
been used.

\subsection{Fourth-order compact time-splitting methods}
To avoid the negative time steps and motivated by the numerical
integration of the Schr\"{o}dinger equation \cite{Chin,CC,CC2},
a {\bf fourth-order gradient symplectic integrator} was
proposed by S. A. Chin \cite{Chin} as
\begin{equation}
\label{4A}
u(\tau, \bx) \approx S_\text{4c}(\tau)u_0(\bx) := e^{\frac{1}{6}\tau W}e^{\frac{1}{2}\tau T}e^{\frac{2}{3}\tau
	\widehat{W}}
e^{\frac{1}{2}\tau T}e^{\frac{1}{6}\tau W}u_0(\bx),
\end{equation}
where
\begin{equation}
\label{WHW1}
\widehat{W} := W + \dfrac{1}{48}\tau^2[W, [T, W]],
\end{equation}
with $[T, W] := TW - WT$ the commutator of the two operators $T$ and $W$
and $[W, [T, W]]$ a double commutator.
Again, by using Taylor expansion, one can formally show the local truncation error as \cite{Chin,CC}
\begin{equation}
\label{err_4c}
\|u(\tau,\bx)-S_\text{4c}(\tau)u_0(\bx)\|\le \widehat{C}_4 \tau^5,
\end{equation}
where $\widehat{C}_4>0$ is a constant independent of $\tau$.
Thus it is formally a fourth-order integrator \cite{MQ}.
In this method, in general,  one needs to integrate the operator $T$ twice and
the operator $W$ three times under the assumption that the computation of
$\widehat{W}$ is equivalent to that of $W$. Thus it is more efficient
than the $S_4$ and $S_\text{4RK}$ methods. In this sense,
it is more appropriate to name it as a {\bf fourth-order compact splitting }
($S_\text{4c}$) method since, at each time step, the number of sub-steps in it  is much less than those in the $S_4$ and $S_\text{4RK}$ methods.
Another advantage of the $S_\text{4c}$ method is that there is no negative time step in it.

\begin{table}[t!]\renewcommand{\arraystretch}{1.2}
	\centering
	\begin{tabular}{|c|ccccc|}
		\hline
		& $S_1$ & $S_2$ & $S_4$ & $S_\text{4RK}$ & $S_\text{4c}$  \\
		\hline
		T & 1 &	1 &	3 &	6 &	2\\
		\hline
		W & 1&	2 &	4 &	7 &	3\\
		\hline
	\end{tabular}
    \caption{The numbers of operators $T$ and $W$ to be implemented in different time-splitting methods.}
    \label{splitting_cmp}
\end{table}

\bigskip

For comparison, Table \ref{splitting_cmp} lists
the numbers of $T$ and $W$ to be integrated
 by different splitting methods.
From it, under the assumption that the computation for $e^{\tau W}$ is easier and/or more efficient than that for $e^{\tau T}$ and the computation
of $e^{\tau \widehat{W}}$ is similar to that for $e^{\tau W}$,
we could draw the following conclusions:
(i) the computational time of $S_2$ is almost the same as that
of $S_1$;
(ii) the computational time of $S_\text{4c}$ is about two times of that
of $S_2$ (or $S_1$);
 (iii) among the three fourth-order splitting methods,
$S_\text{4c}$ is the most efficient and $S_\text{4RK}$ is the most expensive.


\section{Computation for the double commutator $[W, [T, W]]$} \label{section3}
\setcounter{equation}{0}
\setcounter{table}{0}
\setcounter{figure}{0}

In this section, we first show that the double commutator $[W, [T, W]]$
is linear in $T$ and then compute it for the Dirac equations
\eqref{LDirac1d2d} for $d=1,2$ and
\eqref{nondimuni} for $d=1,2,3$.

\begin{lemma}
	\label{lemma_commutator}
	Let $T$ and $W$ be two operators, then we have
	\begin{equation}
	    \label{dcm1}
      	[W, [T, W]] = 2WTW - WWT - TWW.
	\end{equation}
Thus the double commutator $[W, [T, W]]$ is linear in $T$, i.e. for any two
operators $T_1$ and $T_2$, we have
\begin{equation}
	    \label{dcm2}
      	[W, [a_1T_1+a_2T_2, W]] = a_1[W, [T_1, W]]+a_2[W, [T_2, W]],
      \qquad a_1,a_2\in {\mathbb R}.
	\end{equation}
\end{lemma}

\begin{proof}
Noticing $[T, W] := TW - WT$, we have
	\bea
	\label{dcm3}
	   	    [W, [T, W]] &=& [W, (TW - WT)] = W(TW - WT) - (TW - WT)W\nn \\
	   	    &=& WTW - WWT - TWW + WTW\nn\\
	   	    &=& 2WTW - WWT - TWW.
	\eea
From \eqref{dcm3}, it is easy to see that
the double commutator $[W, [T, W]]$ is linear in $T$, i.e.
\eqref{dcm2} is valid.
\end{proof}

\subsection{Double commutators of the Dirac equation in 1D}
\begin{lemma}\label{DTW1d1} For the Dirac equation \eqref{LDirac1d2d} in 1D, i.e. $d = 1$,
define
\begin{equation}
\label{TW1d}
T = -\frac{1}{\varepsilon}\sigma_1
	\partial_1 - \frac{i\nu}{\delta\varepsilon^2}\sigma_3, \qquad
W = -\frac{i}{\delta}\Bigl(V(x)I_2 - A_1(x)\sigma_1\Bigr),
\end{equation}
we have
\begin{equation}
\label{TW1d4}
[W, [T, W]] = -\frac{4i\nu}{\delta^3\eps^2}A_1^2(x)\sigma_3.
\end{equation}
\end{lemma}

\begin{proof}
Combining \eqref{TW1d} and \eqref{dcm2}, we obtain	
	\begin{equation}
\label{TW1d3}
	    [W, [T, W]] = -\frac{1}{\eps}{\left[W, \left[\sigma_1\partial_1, W\right]\right]} - \frac{i\nu}{\delta\eps^2}
	    {\left[W, \left[\sigma_3, W\right]\right]}.
	\end{equation}
Noticing \eqref{dcm1} and \eqref{TW1d}, we have
	\bea
\label{Ts1d1}
		{\left[W, \left[\sigma_1\partial_1, W\right]\right]}&=&2\left(-\frac{i}{\delta}\big(V(x)I_2 - A_1(x)\sigma_1\big)\right)
		\left(\sigma_1\partial_1\right)\left(-\frac{i}{\delta}\big(V(x)I_2 - A_1(x)\sigma_1\big)\right) \nn \\
		 && - \left(-\frac{i}{\delta}\big(V(x)I_2 - A_1(x)\sigma_1\big)\right)^2\left(\sigma_1\partial_1\right)
		- \left(\sigma_1\partial_1\right)\left(-\frac{i}{\delta}\big(V(x)I_2 - A_1(x)\sigma_1\big)\right)^2\nn\\
		&=& -\frac{2}{\delta^2}\big(V(x)I_2 - A_1(x)\sigma_1\big)\sigma_1\partial_1\big(V(x)I_2 - A_1(x)\sigma_1\big)\nn\\
		&& + \frac{1}{\delta^2}\big(V(x)I_2 - A_1(x)\sigma_1\big)^2\sigma_1\partial_1
 + \frac{1}{\delta^2}\sigma_1\partial_1\big(V(x)I_2 - A_1(x)\sigma_1\big)^2\nn\\
		&=& -\frac{2}{\delta^2}\sigma_1\big(V(x)I_2 - A_1(x)\sigma_1\big)\partial_1\big(V(x)I_2 - A_1(x)\sigma_1\big)\nn\\
		&& -\frac{2}{\delta^2}
		\sigma_1\big(V(x)I_2 - A_1(x)\sigma_1\big)^2\partial_1 +\frac{2}{\delta^2}\sigma_1\big(V(x)I_2 - A_1(x)\sigma_1\big)^2\partial_1\nn\\
		&& + \frac{2}{\delta^2}\sigma_1\big(V(x)I_2 - A_1(x)\sigma_1\big)\partial_1\big(V(x)I_2 - A_1(x)\sigma_1\big)\nn\\
	&=& 0.
	\eea
	\bea
\label{Ts1d2}
		{[W, [\sigma_3, W]]}
		&=& 2\left(-\frac{i}{\delta}\big(V(x)I_2 - A_1(x)\sigma_1\big)\right)\sigma_3\left(-\frac{i}{\delta}
		\big(V(x)I_2 - A_1(x)\sigma_1\big)\right)\nn \\ &&-\left(-\frac{i}{\delta}\big(V(x)I_2 - A_1(x)\sigma_1\big)\right)^2\sigma_3
		- \sigma_3\left(-\frac{i}{\delta}\big(V(x)I_2 - A_1(x)\sigma_1\big)\right)^2\nn\\
		&=& -\frac{2}{\delta^2}\big(V(x)I_2 - A_1(x)\sigma_1\big)\big(V(x)I_2 + A_1(x)\sigma_1\big)\sigma_3
		+ \frac{1}{\delta^2}\big(V(x)I_2 - A_1(x)\sigma_1\big)^2\sigma_3\nn\\
        && + \frac{1}{\delta^2}\big(V(x)I_2 + A_1(x)\sigma_1\big)^2\sigma_3\nn\\
		&=& -\frac{1}{\delta^2}\bigg(2V^2(x)I_2 - 2A_1^2(x)I_2 - \big(V^2(x)I_2 + A_1^2(x)I_2 -2A_1(x)V(x)\sigma_1\big)\nn\\
		&& - \big(V^2(x)I_2 + A_1^2(x)I_2 + 2A_1(x)V(x)\sigma_1\big)\bigg)\sigma_3\nn\\
		&=& -\frac{1}{\delta^2}\big(-4A_1^2(x)I_2\big)\sigma_3 = \frac{4}{\delta^2}A_1^2(x)\sigma_3.
	\eea
Plugging \eqref{Ts1d1} and \eqref{Ts1d2} into \eqref{TW1d3},
we can obtain \eqref{TW1d4} immediately.
\end{proof}

Combining \eqref{TW1d4}, \eqref{TW1d} and \eqref{WHW1},
we have
\begin{equation} \label{TW1d7}
\widehat{W}=W + \dfrac{1}{48}\tau^2[W, [T, W]]=
-\frac{i}{\delta}\Bigl(V(x)I_2 - A_1(x)\sigma_1\Bigr)-
\frac{i\nu\tau^2}{12\delta^3\eps^2}A_1^2(x)\sigma_3,
\end{equation}
which immediately implies
that the computation of $e^{\tau \widehat{W}}$ is
similar (or at almost the same computational cost) to that for $e^{\tau W}$ in  this case.

\begin{corollary}
	\label{DTW1d2}
 For the  Dirac equation \eqref{nondimuni} in 1D, i.e. $d = 1$, define
\begin{equation}
T = -\frac{1}{\eps}\alpha_1\partial_1 -
 \frac{i\nu}{\delta\eps^2}\beta, \qquad
 W = -\frac{i}{\delta}\Bigl(V(x)I_4 - A_1(x)\alpha_1\Bigr),
\end{equation}
we have
\begin{equation}
\label{TW1d5}
[W, [T, W]] = -\frac{4i\nu}{\delta^3\eps^2}A_1^2(x)\beta.
\end{equation}
\end{corollary}

\subsection{Double commutators of the Dirac equation in 2D and 3D}

Similar to the 1D case, we have (see detailed computation in Appendix A)

\begin{lemma}
	\label{DTW2d1}
For the Dirac equation \eqref{LDirac1d2d} in 2D, i.e. $d = 2$, define
\begin{equation}
\label{TW2d1}
T = -\frac{1}{\eps}\sigma_1\partial_1 -
	\frac{1}{\eps}\sigma_2\partial_2 - \frac{i\nu}{\delta\eps^2}\sigma_3, \qquad W = -\frac{i}{\delta}
	\Bigl(V(\bx)I_2 - A_1(\bx)\sigma_1 - A_2(\bx)\sigma_2\Bigr),
\end{equation}
we have
\begin{equation}
	\label{commutator_2D}
[W, [T, W]] = F_3(\mathbf{x}) + F_1(\mathbf{x})\partial_1 + F_2(\mathbf{x})\partial_2,
	\end{equation}
where
	\beas
	\label{commutator_2D_F}
	F_1(\mathbf{x}) &=& \frac{4}{\delta^2\eps}\Bigl( - A_2^2(\bx)\sigma_1+A_1(\bx)A_2(\bx)\sigma_2\Bigr),
	\quad F_2(\mathbf{x}) = \frac{4}{\delta^2\eps}\Bigl(A_1(\bx)A_2(\bx)\sigma_1 - A_1^2(\bx)\sigma_2\Bigr), \\
	F_3(\mathbf{x}) &=& \frac{4}{\delta^2\eps}
\Big(A_1(\bx)\partial_2A_2(\bx)-A_2(\bx)\partial_1A_2(\bx)\Big)
\sigma_1+\frac{4}{\delta^2\eps}
\Big(A_2(\bx)\partial_1A_1(\bx)-A_1(\bx)\partial_2A_1(\bx)\Big)
\sigma_2\\
&&+\frac{4i}{\delta^2\eps}\left(A_2(\bx)\partial_1V(\bx)-A_1(\bx)\partial_2 V(\bx)-\frac{\nu}{\delta\eps}
	\big(A_1^2(\bx) + A_2^2(\bx)\big)\right)\sigma_3.
	\eeas
\end{lemma}

\begin{corollary}
	\label{DTW2d2}
	For the Dirac equation \eqref{nondimuni} in 2D, i.e. $d = 2$, define
\begin{equation}
\label{TW2d2}
T = -\frac{1}{\eps}\alpha_1\partial_1 -
	\frac{1}{\eps}\alpha_2\partial_2 - \frac{i\nu}{\delta\eps^2}\beta,
\quad W = -\frac{i}{\delta}
	\Bigl(V(\bx)I_2 - A_1(\bx)\alpha_1 - A_2(\bx)\alpha_2\Bigr),
\end{equation}
we have
\begin{equation}
\label{commutator_2D2}
	[W, [T, W]] = F_3(\mathbf{x})+ F_1(\mathbf{x})\partial_1 + F_2(\mathbf{x})\partial_2,
	\end{equation}
where
	\beas
F_1(\mathbf{x}) &=& \frac{4}{\delta^2\eps}\Bigl( - A_2^2(\bx)\alpha_1+A_1(\bx)A_2(\bx)\alpha_2\Bigr),
	\quad F_2(\mathbf{x}) = \frac{4}{\delta^2\eps}\Bigl(A_1(\bx)A_2(\bx)\alpha_1 - A_1^2(\bx)\alpha_2\Bigr), \\
		F_3(\mathbf{x}) &=& \frac{4}{\delta^2\eps}\Big(A_1(\bx)\partial_2A_2(\bx)-
A_2(\bx)\partial_1A_2(\bx)\Big)\alpha_1+
\frac{4}{\delta^2\eps}\Big(A_2(\bx)\partial_1A_1(\bx)-
A_1(\bx)\partial_2A_1(\bx)\Big)\alpha_2\\
&&+\frac{4i}{\delta^2\eps}\Big(A_2(\bx)\partial_1V(\bx)-
A_1(\bx)\partial_2V(\bx)\Big)\gamma\alpha_3
	- \frac{4i\nu}{\delta^3\eps^2}
	\big(A_1^2(\bx) + A_2^2(\bx)\big)\beta,
	\eeas
where
\be\label{gammam}
\gamma = \begin{pmatrix}  \mathbf{0} &I_2\\ I_2 &\mathbf{0} \end{pmatrix}. \quad
\ee
\end{corollary}

For the Dirac equation (\ref{nondimuni}) in 3D, i.e. $d = 3$, we have
(see detailed computation in Appendix B)

\begin{lemma}
	\label{DTW3d}
For the Dirac equation \eqref{nondimuni} in 3D, i.e. $d = 3$, define
\begin{equation}
\label{TW3d1}
T = -\frac{1}{\eps}\sum_{j = 1}^3\alpha_j\partial_j -\frac{i\nu}{\delta\eps^2} \beta,\quad
W = -\frac{i}{\delta}\Bigl(V(\bx)I_4 - \sum_{j = 1}^3A_j(\bx)\alpha_j\Bigr),
\end{equation}
we have
	\begin{equation}
	\label{commutator_3D}
	[W, [T, W]] = F_4(\mathbf{x})+F_1(\mathbf{x})\partial_1 + F_2(\mathbf{x})\partial_2 + F_3(\mathbf{x})\partial_3,
    \end{equation}
    where
    \beas
    F_1(\mathbf{x}) &=&\frac{4}{\delta^2\eps}\Big(-\big(A_2^2(\bx) + A_3^2(\bx)\big)\alpha_1 +
    A_1(\bx)A_2(\bx)\alpha_2 + A_1(\bx)A_3(\bx)\alpha_3\Big), \\
    F_2(\mathbf{x}) &=&\frac{4}{\delta^2\eps}\Big(
    A_2(\bx)A_1(\bx)\alpha_1-\big(A_1^2(\bx) + A_3^2(\bx)\big)\alpha_2 + A_2(\bx)A_3(\bx)\alpha_3\Big), \\
    F_3(\mathbf{x}) &=&\frac{4}{\delta^2\eps}\Big(
    A_3(\bx)A_1(\bx)\alpha_1 + A_3(\bx)A_2(\bx)\alpha_2-\big(A_1^2(\bx) + A_2^2(\bx)\big)\alpha_3\Big), \\
    F_4(\mathbf{x}) &=&\frac{4}{\delta^2\eps}
    \Big(A_1(\bx)\big(\partial_2A_2(\bx)+\partial_3A_3(\bx)\big)-
    A_2(\bx)\partial_1A_2(\bx)-A_3(\bx)\partial_1A_3(\bx)
    \Big)\alpha_1\\
    &&+\frac{4}{\delta^2\eps}
    \Big(A_2(\bx)\big(\partial_1A_1(\bx)+\partial_3A_3(\bx)\big)-
    A_1(\bx)\partial_2A_1(\bx)-A_3(\bx)\partial_2A_3(\bx)
    \Big)\alpha_2\\
    &&+\frac{4}{\delta^2\eps}
    \Big(A_3(\bx)\big(\partial_1A_1(\bx)+\partial_2A_2(\bx)\big)-
    A_1(\bx)\partial_3A_1(\bx)-A_2(\bx)\partial_3A_2(\bx)
    \Big)\alpha_3\\
    &&+\frac{4i}{\delta^2\eps}
    \Big(A_1(\bx)\big(\partial_2A_3(\bx)-\partial_3A_2(\bx)\big)+
    A_2(\bx)\big(\partial_3A_1(\bx)-\partial_1A_3(\bx)\big)\\
    &&+   A_3(\bx)\big(\partial_1A_2(\bx)-\partial_2A_1(\bx)\big)
    \Big)\gamma + \frac{4i}{\delta^2\eps}\Big(A_3(\bx)\partial_2V(\bx) - A_2(\bx)\partial_3V(\bx)\Big)\gamma\alpha_1\\
    && + \frac{4i}{\delta^2\eps}\Big(A_1(\bx)\partial_3V(\bx) - A_3(\bx)\partial_1V(\bx)\Big)\gamma\alpha_2\\
    &&+ \frac{4i}{\delta^2\eps}\Big(A_2(\bx)\partial_1V(\bx) - A_1(\bx)\partial_2V(\bx)\Big)\gamma\alpha_3
    - \frac{4i\nu}{\delta^3\eps^2}\Big(A_1^2(\bx) + A_2^2(\bx) + A_3^2(\bx)\Big)\beta.
    \eeas
\end{lemma}

From Lemmas \ref{DTW1d1}, \ref{DTW2d1} and \ref{DTW3d} and Corollaries \ref{DTW1d2} and \ref{DTW2d2}, it is easy to observe that the double commutator will vanish
when the Dirac
equation \eqref{LDirac1d2d} (or \eqref{nondimuni}) has no magnetic potentials.

\begin{lemma}
	\label{DTW_nonmag}
For the Dirac equation \eqref{LDirac1d2d} in 1D and 2D, and for the Dirac equation \eqref{nondimuni} in 1D, 2D and 3D, when there is no magnetic potential, i.e., when $A_1(\bx) = A_2(\bx) = A_3(\bx) \equiv 0$, we have
	\begin{equation}
	\label{DC_nonmag}
	[W, [T, W]] = 0.
	\end{equation}
\end{lemma}


\section{A fourth-order compact time-splitting Fourier pseudospectral method}
In this section, we present a fourth-order compact time-splitting Fourier pseudospectral method for the Dirac equation \eqref{nondimuni} (or (\ref{LDirac1d2d})) by using the $S_\text{4c}$ method \eqref{4A} for time integration followed by the Fourier pseudospectral method for spatial
discretization.

\subsection{Time integration by the $S_\text{4c}$ method in 1D}
For simplicity of notations, we present the numerical method  for
(\ref{LDirac1d2d}) in 1D first.
Similar to most works in the literatures for the analysis
and computation of the Dirac equation (cf. \cite{BCJT,BCJT2,BCJY,BL} and references therein),
in practical computation, we truncate the whole space problem onto an interval $\Omega=(a,b)$
with periodic boundary conditions. The truncated interval is large enough such that the truncation error is
negligible. In 1D, the Dirac equation (\ref{LDirac1d2d}) with periodic boundary conditions collapses to
\be
\begin{split}
	\label{LDb1}
	&i\delta\partial_t\Phi=\left(-i\dfrac{\delta}{\eps}\sigma_1\partial_x + \dfrac{\nu}{\eps^2}
	\sigma_3\right)\Phi + \Bigl(V(x)I_2 - A_1(x)\sigma_1\Bigr)\Phi, \quad x\in\Omega, \quad t>0,\\
&\Phi(t,a)=\Phi(t,b), \quad \partial_x\Phi(t,a)=\partial_x\Phi(t,b),\quad
t\ge0;\\
 &\Phi(0,x)=\Phi_0(x), \quad a\le x\le b;
 \end{split}
	\ee
where $\Phi:=\Phi(t,x)$, $\Phi_0(a)=\Phi_0(b)$ and $\Phi_0^\prime(a)=\Phi_0^\prime(b)$.

Choose a time step $\tau>0$, denote $t_n=n\tau$ for $n\ge0$ and
let $\Phi^n(x)$ be an approximation of $\Phi(t_n,x)$.  Re-writing the Dirac equation \eqref{LDb1} as
\begin{equation}
	\label{LDb2}
	\partial_t\Phi=\left(-\dfrac{1}{\eps}\sigma_1\partial_x - \dfrac{i\nu}{\delta\eps^2}
	\sigma_3\right)\Phi -\frac{i}{\delta} \Bigl(V(x)I_2 - A_1(x)\sigma_1\Bigr)\Phi:=(T+W)\Phi,
\end{equation}
then we can apply the $S_\text{4c}$ method \eqref{4A} for time integration
over the time interval $[t_n,t_{n+1}]$ as
\begin{equation}
\label{Dirac4A}
\Phi^{n+1}(x)= S_\text{4c}(\tau)\Phi^n(x) := e^{\frac{1}{6}\tau W}e^{\frac{1}{2}\tau T}e^{\frac{2}{3}\tau
	\widehat{W}}
e^{\frac{1}{2}\tau T}e^{\frac{1}{6}\tau W}\Phi^n(x), \quad
a\le x\le b, \quad n\ge0,
\end{equation}
where the two operators $T$ and $W$ are given in \eqref{TW1d} and
the operator $\widehat{W}$ is given in \eqref{TW1d7}.
In order to calculate $e^{\frac{1}{2}\tau T}$, we can discretize it in space via Fourier
spectral method and then integrate (in phase space or Fourier space)
in time {\bf exactly} \cite{BCJT2,BL}. Since $W$ is diagonalizable \cite{BCJT2},
$e^{\frac{1}{6}\tau W}$ can be evaluated very efficiently \cite{BCJT2}. For $e^{\frac{2}{3}\tau \widehat{W}}$,
plugging \eqref{Pauli} into \eqref{TW1d7}, we can diagonalize it as
\begin{equation} \label{TW1d8}
\widehat{W}=
-\frac{i}{\delta}\Bigl(V(x)I_2 - A_1(x)\sigma_1\Bigr)-
\frac{i\nu\tau^2}{12\delta^3\eps^2}A_1^2(x)\sigma_3
=-iP_2(x)\,\Lambda_2(x)\, P_2(x)^*:=\widehat{W}(x),
\end{equation}
where $\Lambda_2(x)={\rm diag}(\lambda_+^{(2)}(x),\lambda_-^{(2)}(x))$ with
$\lambda_\pm^{(2)}(x)=\frac{V(x)}{\delta}\pm \frac{A_1(x)}{12\delta^3\eps^2}
\sqrt{144\delta^4\eps^4 + \nu^2\tau^4A_1^2(x)}$ and
\begin{equation}
P_2(x)=\frac{1}{\sqrt{2\beta_1(x)}}\left(\ba{ll}
\sqrt{\beta_1(x) + \beta_2(x)} & \sqrt{\beta_1(x) - \beta_2(x)}\\
-\sqrt{\beta_1(x) - \beta_2(x)} & \sqrt{\beta_1(x) + \beta_2(x)} \\
 \ea
 \right), \quad a\le x\le b,
\end{equation}
with
\begin{equation}
\beta_1(x) = \sqrt{144\delta^4\eps^4 + \nu^2\tau^4A_1^2(x)}, \quad
\beta_2(x) = \nu\tau^2A_1(x), \quad a\le x\le b.
\end{equation}
Thus we have
\begin{equation}
e^{\frac{2}{3}\tau\widehat{W}}=e^{-\frac{2i}{3}\tau P_2(x)\,\Lambda_2(x)\, P_2(x)^*}=P_2(x)\,e^{-\frac{2i}{3}\tau\Lambda_2(x)} \, P_2(x)^*, \quad a\le x\le b.
\end{equation}

\subsection{Full discretization in 1D}
Choose a mesh size $h:=\Delta x =\frac{b-a}{M}$ with $M$ being an even positive integer and denote the grid points as
$x_j:=a+jh$, for  $j=0, 1, \ldots ,M$.
Denote $X_M=\{\beU=(U_0,U_1,...,U_M)^T\ |\ U_j\in {\mathbb C}^2, j=0,1,\ldots,M, \ U_0=U_M\}$.
For any $\beU\in X_M$,
we denote its Fourier representation as
\be
U_j=\sum_{l=-M/2}^{M/2-1} \widetilde{U}_l\, e^{i\mu_l(x_j-a)}
=\sum_{l=-M/2}^{M/2-1} \widetilde{U}_l\, e^{2ijl\pi/M}, \qquad j=0,1,\ldots, M,
\end{equation}
where $\mu_l$ and $\widetilde{U}_l\in{\mathbb C}^2$  are defined as
\be\mu_l=\frac{2l\pi}{b-a},\qquad \widetilde{U}_l=\frac{1}{M}\sum_{j=0}^{M-1}U_j\, e^{-2ijl\pi/M}, \qquad
 l=-\frac{M}{2}, \ldots,\frac{M}{2}-1.\label{dftco}
\ee
For $\beU\in X_M$ and $u(x)\in L^2(\Omega)$, their $l^2$-norms are defined as
\begin{eqnarray}
\|\beU\|^2_{l^2}:=h\sum^{M-1}_{j=0}|U_j|^2, \qquad  \|u\|^2_{l^2}:=h\sum^{M-1}_{j=0}|u(x_j)|^2.
\end{eqnarray}
Let $\Phi_{j}^{n}$  be the numerical approximation of $\Phi(t_n,x_j)$ and
denote $\Phi^n=\left(\Phi_0^n, \Phi_1^n, \ldots, \Phi_M^n\right)^T\in X_M$ as
the solution vector at $t=t_n$. Take $\Phi_j^0 = \Phi_0(x_j)$ for $j = 0, \ldots, M$, then  a {\bf fourth-order compact time-splitting Fourier pseudospectral} ($S_\text{4c}$) discretization for the Dirac equation \eqref{LDb1}
is given as
	\be\begin{split}
	\label{S4c}
	\Phi_j^{(1)} &= e^{\frac{\tau}{6}W(x_j)}\Phi_j^{n} = P_{1}\,e^{-\frac{i\tau}{6}\Lambda_{1}(x_j)}P_1^*\,\Phi_j^{n},\\
    \Phi_j^{(2)}&=\sum_{l=-M/2}^{M/2-1} e^{\tau \Gamma_l}\,
	{\left(\widetilde{\Phi^{(1)}}\right)}_l\, e^{i\mu_l(x_j-a)}
	=\sum_{l=-M/2}^{M/2-1}Q_l\, e^{-i\tau D_l}\,Q_l^*\,{\left(\widetilde{\Phi^{(1)}}
		\right)}_l\, e^{2ijl\pi/M},\\
	\Phi_j^{(3)}&=e^{\frac{2\tau}{3}\widehat{W}(x_j)}\Phi_j^{(2)}=P_2
(x_j)\,e^{-\frac{2i\tau}{3}\Lambda_2(x_j)}\,P_2(x_j)^*\,\Phi_j^{(2)},\qquad j = 0,1, \ldots, M,\\
	\Phi_j^{(4)}&=\sum_{l=-M/2}^{M/2-1} e^{\tau\Gamma_l}\,
	{\left(\widetilde{\Phi^{(3)}}\right)}_l\, e^{i\mu_l(x_j-a)}=
	\sum_{l=-M/2}^{M/2-1} Q_l\,e^{-i\tau D_l}\,Q_l^*\,{\left(\widetilde{\Phi^{(3)}}\right)}_l\, e^{2ijl\pi/M},\\
	\Phi_j^{n + 1}& = e^{\frac{\tau}{6}W(x_j)}\Phi_j^{(4)} = P_1\,e^{-\frac{i\tau}{6}\Lambda_1(x_j)}\,P_1^*\,
	\Phi_j^{(4)},
\end{split}
	\ee
	where
	\be
\begin{split}
&W(x_j):=-\frac{i}{\delta} \Bigl(V(x_j)I_2 - A_1(x_j)\sigma_1\Bigr)
=-iP_{1}\,\Lambda_{1}(x_j)\,P_{1}^*,\quad j=0,1,\ldots,M,\\
&\Gamma_l = -\frac{i\mu_l}{\eps}\sigma_1 - \frac{i\nu}{\delta\eps^2}\sigma_3
	= -iQ_l\,D_l\,Q_l^*, \quad l = -\dfrac{M}{2}, \ldots, \dfrac{M}{2} - 1,
\end{split}
	\ee
with $D_l={\rm diag}\left(\frac{1}{\delta \eps^2}\sqrt{\nu^2+\delta^2\eps^2 \mu_l^2}, -\frac{1}{\delta \eps^2}\sqrt{\nu^2+\delta^2\eps^2 \mu_l^2}\right)$,
$\Lambda_1(x)={\rm diag}\left(\lambda_+^{(1)}(x),\lambda_-^{(1)}(x)\right)$ with
$\lambda_\pm^{(1)}(x)=\frac{1}{\delta}\big(V(x)\pm A_1(x)\big)$,
$\eta_l = \sqrt{\nu^2 + \delta^2\eps^2\mu_l^2}$, and
\begin{equation}
P_1=\left(\ba{cc}
\frac{1}{\sqrt{2}} & \frac{1}{\sqrt{2}} \\
-\frac{1}{\sqrt{2}} & \frac{1}{\sqrt{2}}\\
 \ea
 \right), \quad
Q_l=\frac{1}{\sqrt{2\eta_l(\eta_l + \nu)}}\left(\ba{ll}
\eta_l + \nu & -\delta\eps\mu_l\\
\delta\eps\mu_l & \eta_l + \nu\\
\ea \right), \quad l = -\dfrac{M}{2}, \ldots, \dfrac{M}{2} - 1.
\end{equation}

 We remark here that full discretization by other time-splitting methods
together with Fourier pseudospectral method for spatial discretization
can be implemented similarly \cite{BCJT2} and the details are omitted here for
brevity.

	\subsection{Mass conservation in 1D}
The $S_\text{4c}$ method \eqref{S4c} is explicit, its memory cost is $O(M)$ and its computational cost
per time step is $O(M\, \ln M)$, it is fourth-order accurate in time and
spectral accurate in space. In addition, it conserves the total probability
in the discretized level, as shown in the following lemma.
	
	\begin{lemma}
For any $\tau>0$, the $S_\text{4c}$ method (\ref{S4c}) conserves the mass in the discretized level, i.e.
		\begin{equation} \label{nclema1}
		\Big\|\Phi^{n+1}\Big\|^2_{l^2} := h\sum_{j = 0}^{M - 1}\Big|\Phi_j^{n+1}\Big|^2 \equiv h\sum_{j = 0}^{M -
			1}\Big|\Phi_j^0\Big|^2 = h\sum_{j = 0}^{M - 1}\Big|\Phi_0(x_j)\Big|^2=\Big\|\Phi_0\Big\|^2_{l^2},\quad n\ge0.
		\end{equation}
	\end{lemma}

    \begin{proof}
 Noticing
$W(x_j)^*=-W(x_j)$ and thus $\left(e^{\frac{\tau}{6}W(x_j)}\right)^*\,e^{\frac{\tau}{6}W(x_j)}=I_2$,
from (\ref{S4c}) and summing for $j=0,1,\ldots,M-1$, we get
    	\begin{eqnarray}\label{nc41}
    \Big\|\Phi^{n+1}\Big\|^2_{l^2} &=& h\sum_{j = 0}^{M - 1}\Big|\Phi_j^{n+1}\Big|^2=h\sum_{j = 0}^{M - 1}\left|e^{\frac{\tau}{6}W(x_j)}\Phi_j^{(4)}\right|^2=
    h\sum_{j = 0}^{M - 1}(\Phi_j^{(4)})^*\,\left(e^{\frac{\tau}{6}W(x_j)}\right)^*\,
    e^{\frac{\tau}{6}W(x_j)}\,\Phi_j^{(4)} \nonumber\\
    &=&h\sum_{j = 0}^{M - 1}(\Phi_j^{(4)})^*\,I_2\,\Phi_j^{(4)}=h\sum_{j = 0}^{M - 1}\Big|\Phi_j^{(4)}\Big|^2=\Big\|\Phi^{(4)}\Big\|^2_{l^2}, \qquad n\ge0.
    	\end{eqnarray}
Similarly, we have
\be\label{nc4a2}
\Big\|\Phi^{(3)}\Big\|^2_{l^2} =\Big\|\Phi^{(2)}\Big\|^2_{l^2},\qquad
\Big\|\Phi^{(1)}\Big\|^2_{l^2} =\Big\|\Phi^{n}\Big\|^2_{l^2}, \qquad n\ge0.
\ee
Similarly, using the Parsval's identity and noticing $\Gamma_l^*=-\Gamma_l$
and thus $\left(e^{\tau \Gamma_l}\right)^*e^{\tau \Gamma_l}=I_2$,
we get
\be\label{nc4a3}
\Big\|\Phi^{(4)}\Big\|^2_{l^2} =\Big\|\Phi^{(3)}\Big\|^2_{l^2},\qquad
\Big\|\Phi^{(2)}\Big\|^2_{l^2} =\Big\|\Phi^{(1)}\Big\|^2_{l^2}.
\ee
Combining \eqref{nc41}, \eqref{nc4a2} and \eqref{nc4a3}, we obtain
\be
\Big\|\Phi^{n+1}\Big\|^2_{l^2}=\Big\|\Phi^{(4)}\Big\|^2_{l^2}
=\Big\|\Phi^{(3)}\Big\|^2_{l^2}=\Big\|\Phi^{(2)}\Big\|^2_{l^2}
=\Big\|\Phi^{(1)}\Big\|^2_{l^2}=\Big\|\Phi^{n}\Big\|^2_{l^2}, \quad n\ge0.
\ee
Using the mathematical induction, we get the mass conservation \eqref{nclema1}.
\end{proof}

\subsection{Discussion on extension to 2D and 3D}

When there is no magnetic potential, i.e., when $A_1(\bx) = A_2(\bx) = A_3(\bx) \equiv 0$ in the Dirac equation \eqref{LDirac1d2d} in 2D and \eqref{nondimuni} in 2D and 3D, from Lemma \ref{DTW_nonmag}, we know that
the double commutator $[W, [T, W]] = 0$.
 In this case, noting \eqref{WHW1}, we  have
\begin{equation}
\widehat{W} = W + \frac{1}{48}\tau^2[W, [T, W]] = W.
\end{equation}
Then the $S_\text{4c}$ method \eqref{4A} collapses to
\begin{equation}
\label{4Am1}
u(\tau, \bx) \approx S_\text{4c}(\tau)u_0(\bx) := e^{\frac{1}{6}\tau W}e^{\frac{1}{2}\tau T}e^{\frac{2}{3}\tau W}
e^{\frac{1}{2}\tau T}e^{\frac{1}{6}\tau W}u_0(\bx).
\end{equation}
Applying the $S_\text{4c}$ method \eqref{4Am1} to integrate
the Dirac equation \eqref{LDirac1d2d} in 2D
over the time interval $[t_n,t_{n+1}]$ with $\Phi(t_n,\bx)=\Phi^n(\bx)$ given, we obtain
\begin{equation}\label{D2dnm}
\Phi^{n+1}(\bx)= S_\text{4c}(\tau)\Phi^n(\bx)
= e^{\frac{1}{6}\tau W}e^{\frac{1}{2}\tau T}e^{\frac{2}{3}\tau W}e^{\frac{1}{2}\tau T}e^{\frac{1}{6}\tau W}\Phi^n(\bx), \quad
\bx\in\Omega, \quad n\ge0,
\end{equation}
where $T$ and $W$ are given in \eqref{TW2d1}.
Similarly, applying the $S_\text{4c}$ method \eqref{4Am1} to integrate
the Dirac equation \eqref{nondimuni}  in 2D and 3D
over the time interval $[t_n,t_{n+1}]$ with $\Psi(t_n,\bx)=\Psi^n(\bx)$ given, we obtain
\begin{equation} \label{D3dnm}
\Psi^{n+1}(\bx)= S_\text{4c}(\tau)\Psi^n(\bx)
= e^{\frac{1}{6}\tau W}e^{\frac{1}{2}\tau T}e^{\frac{2}{3}\tau W}e^{\frac{1}{2}\tau T}e^{\frac{1}{6}\tau W}\Psi^n(\bx), \quad
\bx\in\Omega, \quad n\ge0,
\end{equation}
where $T$ and $W$ are given in \eqref{TW2d2} and \eqref{TW3d1} for
2D and 3D, respectively. In practical computation,
the operators $e^{\frac{1}{6}\tau W}$ and $e^{\frac{2}{3}\tau W}$
in \eqref{D2dnm} and \eqref{D3dnm} can be evaluated in physical space
directly and easily \cite{BCJT2}. For the operator
$e^{\frac{1}{2}\tau T}$,  it can be discretized  in space via Fourier
spectral method and then integrate (in phase space or Fourier space)
in time {\bf exactly}. For details, we refer to \cite{BCJT2,BL} and
references therein. In fact, the implementation of the $S_\text{4c}$ method
in this case is much simpler than that of the $S_\text{4}$ and $S_\text{4RK}$
methods.

  Of course, when the magnetic potential is nonzero in the Dirac equation \eqref{LDirac1d2d} in 2D and \eqref{nondimuni} in 2D and 3D, one has to adapt the formulation \eqref{4Am1} for $S_\text{4c}$ method. In this case,
  the main difficulty is how to efficiently and accurately evaluate the operator
  $e^{\frac{2}{3}\tau \hat{W}}$. This can be done by using the method of characteristics and the nonuniform fast Fourier transform (NUFFT), which
  has been developed for the magnetic Schr\"{o}dinger equation.
  For details, we refer to \cite{COP,JGB} and references therein.
  Of course, it is a little more tedious in practical implementation
  for $S_\text{4c}$ method than that for the $S_\text{4}$  and $S_\text{4RK}$ methods in this situation.


\section{Comparision of different time-splitting methods}\label{sec5}
\setcounter{equation}{0}
\setcounter{table}{0}
\setcounter{figure}{0}

In this section, we compare  the
fourth-order compact time-splitting Fourier pseudospectral
$S_\text{4c}$ method \eqref{S4c} with other time-splitting
methods including the first-order time-splitting ($S_1$) method,
the second-order time-splitting ($S_2$) method,
the fourth-order time-splitting ($S_4$) method and the fourth-order
partitioned Runge-Kutta time-splitting ($S_\text{4RK}$) method in terms of accuracy and efficiency as well as long time behavior.

\subsection{An example in 1D}

For simplicity, we first consider an example in 1D.
In the  Dirac equation (\ref{LDirac1d2d}), we take $d=1$, $\eps = \delta = \nu = 1$ and
\begin{equation}
\label{potentials}
V(x) = \dfrac{1 - x}{1 + x^2}, \quad A_1(x) = \dfrac{(x + 1)^2}{1 + x^2}, \quad x\in{\mathbb R}.
\end{equation}
The initial data in \eqref{initial11} is taken as:
\begin{equation}
\label{initial2}
\phi_1(0, x) = e^{-x^2 / 2}, \quad \phi_2(0, x) = e^{-(x - 1)^2 / 2}, \quad x\in{\mathbb R}.
\end{equation}
The problem is solved numerically on a bounded domain
$\Omega = (-32, 32)$, i.e. $a=-32$ and $b=32$.

Due to the fact that the exact solution is not available,
we obtain a numerical `exact' solution
by using the $S_\text{4c}$ method with a
fine mesh size $h_e = \frac{1}{16}$ and a
small time step $\tau_e = 10^{-5}$.
Let $\Phi^n$ be the numerical solution obtained
by a numerical method with mesh size $h$ and time step $\tau$.
Then the error is quantified as
\begin{equation}
\label{err}
e_\Phi(t_n) = {\|\Phi^n - \Phi(t_n, \cdot)\|_{l^2}}
={\sqrt{h\sum_{j = 0}^{M - 1}|\Phi(t_n, x_j) - \Phi_j^n|^2}}.
\end{equation}

In order to compare the spatial errors, we take time step $\tau=\tau_e =  10^{-5}$
such that the temporal discretization error could be negligible.
Table \ref{spatial_error} lists numerical errors $e_\Phi(t=6)$
for different time-splitting methods under different mesh size $h$.
We remark here that, for the $S_1$ method, in order to observe the spatial error when the mesh size $h=h_0/2^3$, one has to choose time step
$\tau\le 10^{-10}$ which is too small and thus the error is not shown in the table for this case.
From Table \ref{spatial_error}, we could see that all the numerical methods
are spectral order accurate in space (cf. each row in Table \ref{spatial_error}).

\begin{table}[t!]\renewcommand{\arraystretch}{1.2}
	\centering
	\begin{tabular}{|c|cccc|}
		\hline
		&  $h_0=1$ & $h_0 / 2$ & $h_0 / 2^2$ & $h_0 / 2^3$  \\
		\hline
		$S_1$  & 1.01 & 5.16E-2 & 7.07E-5 & -- \\
		\hline
		$S_2$  & 1.01 & 5.16E-2 & 6.96E-5 & 1.92E-10\\
		\hline
		$S_4$  & 1.01 & 5.16E-2 & 6.96E-5 &3.52E-10\\
		\hline
		$S_\text{4c}$  & 1.01 & 5.16E-2 & 6.96E-5 & 3.06E-10\\
		\hline
		$S_\text{4RK}$  & 1.01 & 5.16E-2 & 6.96E-5 & 5.15E-10\\
		\hline
	\end{tabular}
\caption{Spatial errors $e_\Phi(t=6)$ of different time-splitting methods under different mesh size $h$ for the  Dirac equation (\ref{LDirac1d2d}) in 1D. }\label{spatial_error}
\end{table}

In order to compare the temporal errors, we take mesh size  $h=h_e = \frac{1}{16}$
such that the spatial discretization error could be negligible.
Table \ref{Table 1} lists numerical errors $e_\Phi(t=6)$
for different time-splitting methods under different time step $\tau$. In the table, we use second (s) as the unit for CPU time.
For comparison, Figure \ref{cmp} plots $e_\Phi(t=6)$
and  $e_\Phi(t=6)/\tau^\alpha$ with $\alpha$
taken as the order of accuracy of a certain numerical method (in order to
show the constants $C_1$ in \eqref{err_S1}, $C_2$ in \eqref{err_S2}, $C_4$ in \eqref{err_S4},
$\widetilde{C}_4$ in \eqref{err_4RK} and $\widehat{C}_4$ in \eqref{err_4c}) for
different time-splitting methods under different time step $\tau$.

 \begin{table}[t!]\renewcommand{\arraystretch}{1.6}
	\small
	\centering
	\begin{tabular}{|c|c|ccccccc|}
		\hline
		\multicolumn{2}{|c|}{}&  $\tau_0 = 1 / 2$ & $\tau_0 / 2$ & $\tau_0 / 2^2$ &
		$\tau_0 / 2^3$ & $\tau_0 / 2^4$ & $\tau_0 / 2^5$  & $\tau_0 / 2^6$ \\
		\hline
		\multirow{3}{*}{$S_1$}&  $e_\Phi(t=6)$ &  1.17 &	4.71E-1& 2.09E-1 &	9.90E-2 &	4.82E-2 &	2.38E-2 &	1.18E-2 \\
		&rate  & --	 & 	1.31 &	1.17 &	1.08 &	1.04 &	1.02 &	1.01 \\
		&CPU Time  &  0.02 &	0.05 &	0.11 &	0.16 &	0.37 &	0.62 &	\textbf{1.31} \\
		\hline
		\multirow{3}{*}{$S_2$} & $e_\Phi(t=6)$  &  	7.49E-1 &	1.87E-1 &	4.66E-2 &	1.16E-2
		&	2.91E-3 &	7.27E-4	 & 1.82E-4 \\
		& rate  & -- & 	2.00 &	2.00 &	2.00 &	2.00 &	2.00 &	2.00 \\
		& CPU Time  &   0.04 & 0.06 &	0.11 & 0.21 & 0.37 & 0.75 &	\textbf{1.42} \\
		\hline
		\multirow{3}{*}{$S_4$}&$e_\Phi(t=6)$  &   3.30E-1&	3.73E-2 &	3.05E-3 &	2.07E-4 &
		1.32E-5	& 8.29E-7 & 5.20E-8 \\
		&rate  & -- &	3.15 & 	3.61 &	3.89 &	3.97 &	3.99 &	4.00 \\
		&CPU Time  & 	0.10 &	0.16 &	0.38 &	0.58 &	1.09 &	2.23 &	\textbf{4.41} \\
		\hline
		\multirow{3}{*}{$S_\text{4c}$} & $e_\Phi(t=6)$  &  1.66E-2 &	9.54E-4 &	5.90E-5 &	3.68E-6
		&	2.30E-7 &	1.43E-8 &	8.12E-10 \\
		& rate  & -- & 	4.12 &	4.01 &	4.00 &	4.00 &	4.01 &	4.13 \\
		& CPU Time  &  	0.06 &	0.09 &	0.18 &	0.35 &	0.68 &	1.36 &	\textbf{2.68} \\
		\hline
		\multirow{3}{*}{$S_\text{4RK}$} & $e_\Phi(t=6)$  &   2.87E-3 & 1.78E-4 & 1.11E-5 & 6.97E-7
		& 4.34E-8 & 2.58E-9 & 1.66E-10 \\
		& rate  & -- &  4.01 & 3.99 &	4.00 & 4.00 & 4.07 & 3.96 \\
		& CPU Time  &  0.15 & 0.28 & 0.57 & 1.24 & 2.66 & 3.94 &	\textbf{7.79} \\
	    \hline
	\end{tabular}
\caption{Temporal errors $e_\Phi(t=6)$ of different time-splitting methods under different time step $\tau$ for the  Dirac equation (\ref{LDirac1d2d}) in 1D. Here we also list convergence rates and computational time  (CPU time in seconds) for comparison.}
	\label{Table 1}
\end{table}

\begin{figure}[h!]
\centerline{\psfig{figure=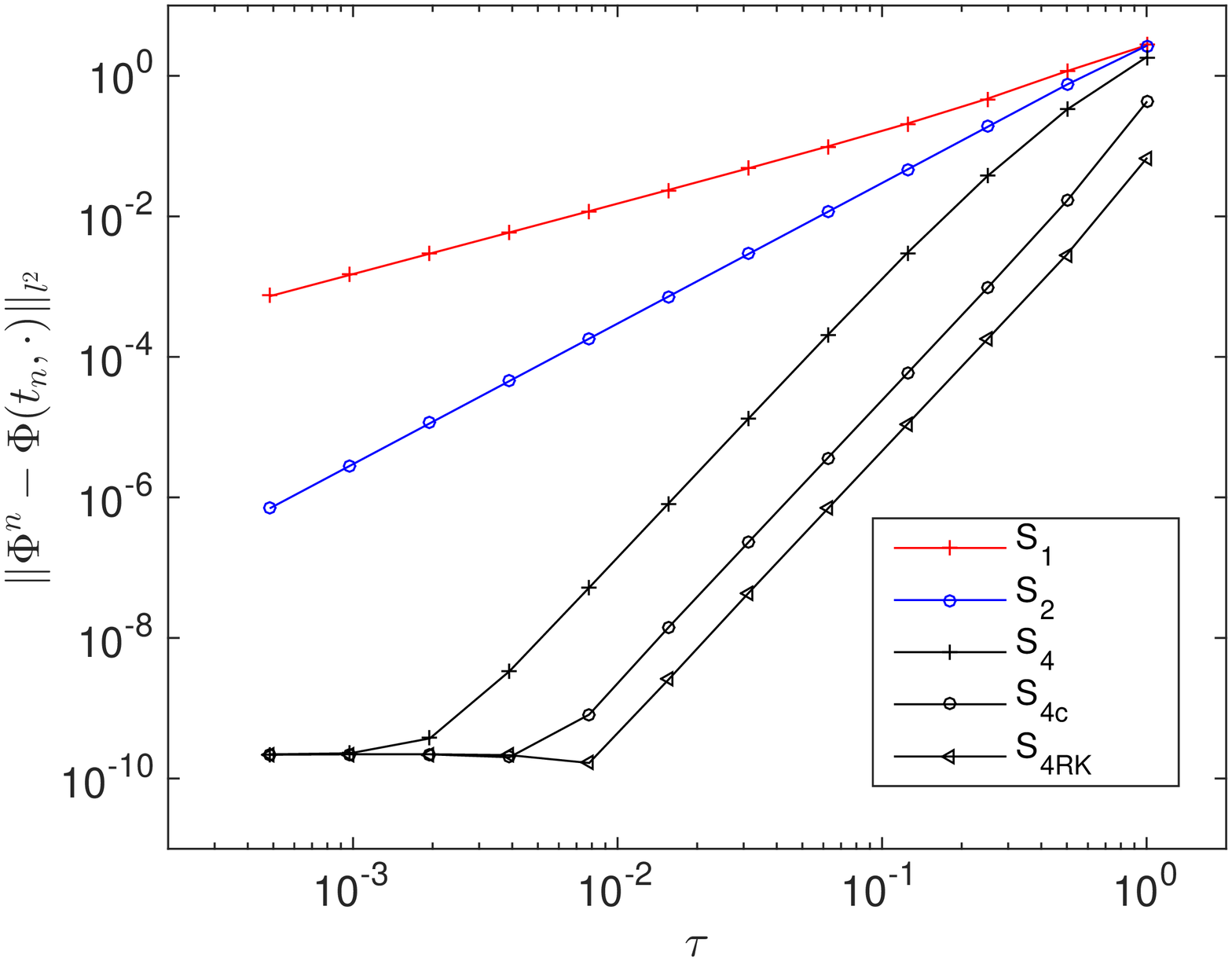,height=6.5cm,width=6.5cm}
\psfig{figure=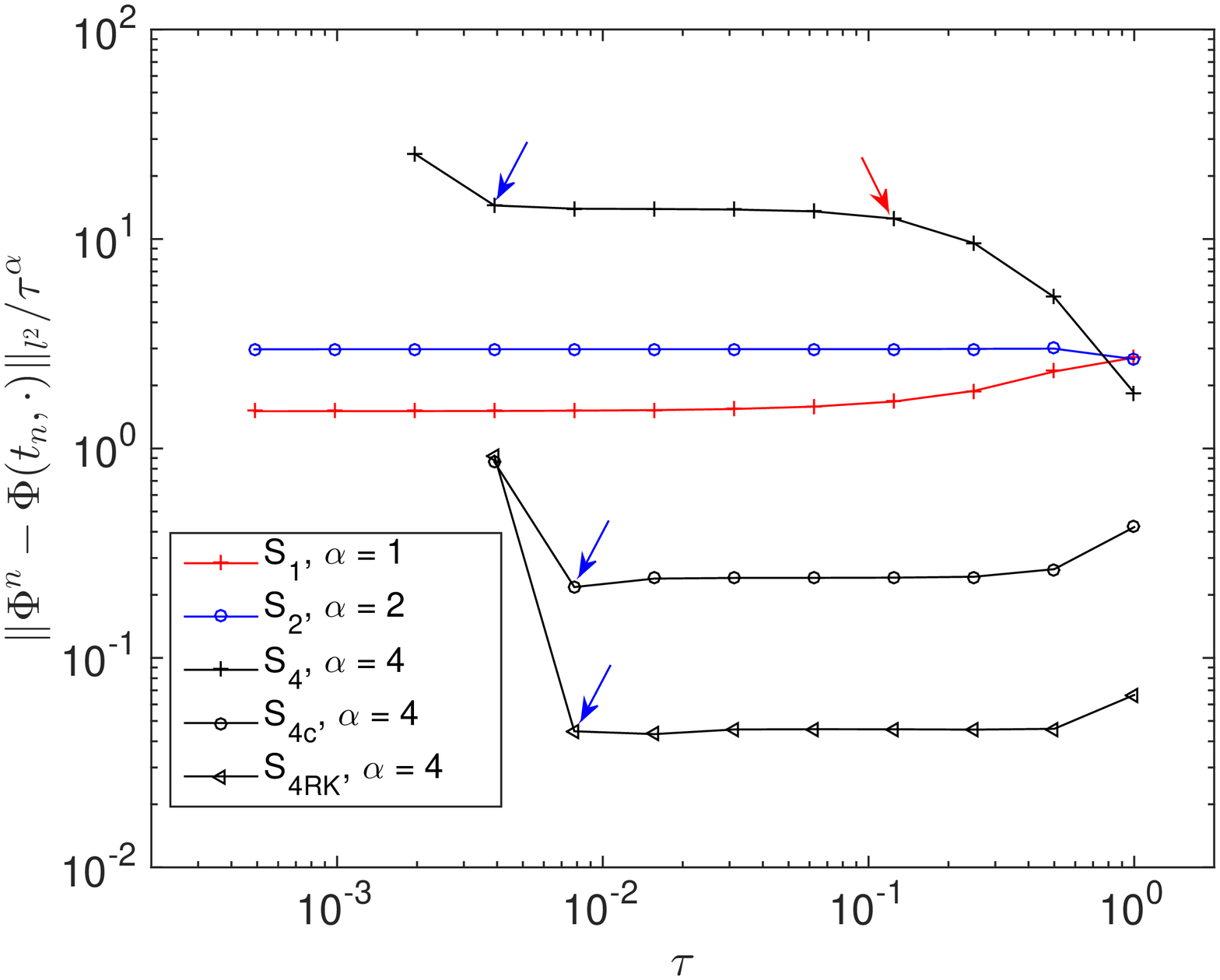,height=6.5cm,width=6.5cm}}
	\caption{Temporal errors $e_\Phi(t=6)$ (left) and $e_\Phi(t=6)/\tau^\alpha$ with $\alpha$
taken as the order of accuracy of a certain numerical method (right)
of different time-splitting methods under different time step $\tau$
for the  Dirac equation (\ref{LDirac1d2d}) in 1D.}
	\label{cmp}
\end{figure}

\begin{figure}[htbp]
\centerline{\psfig{figure=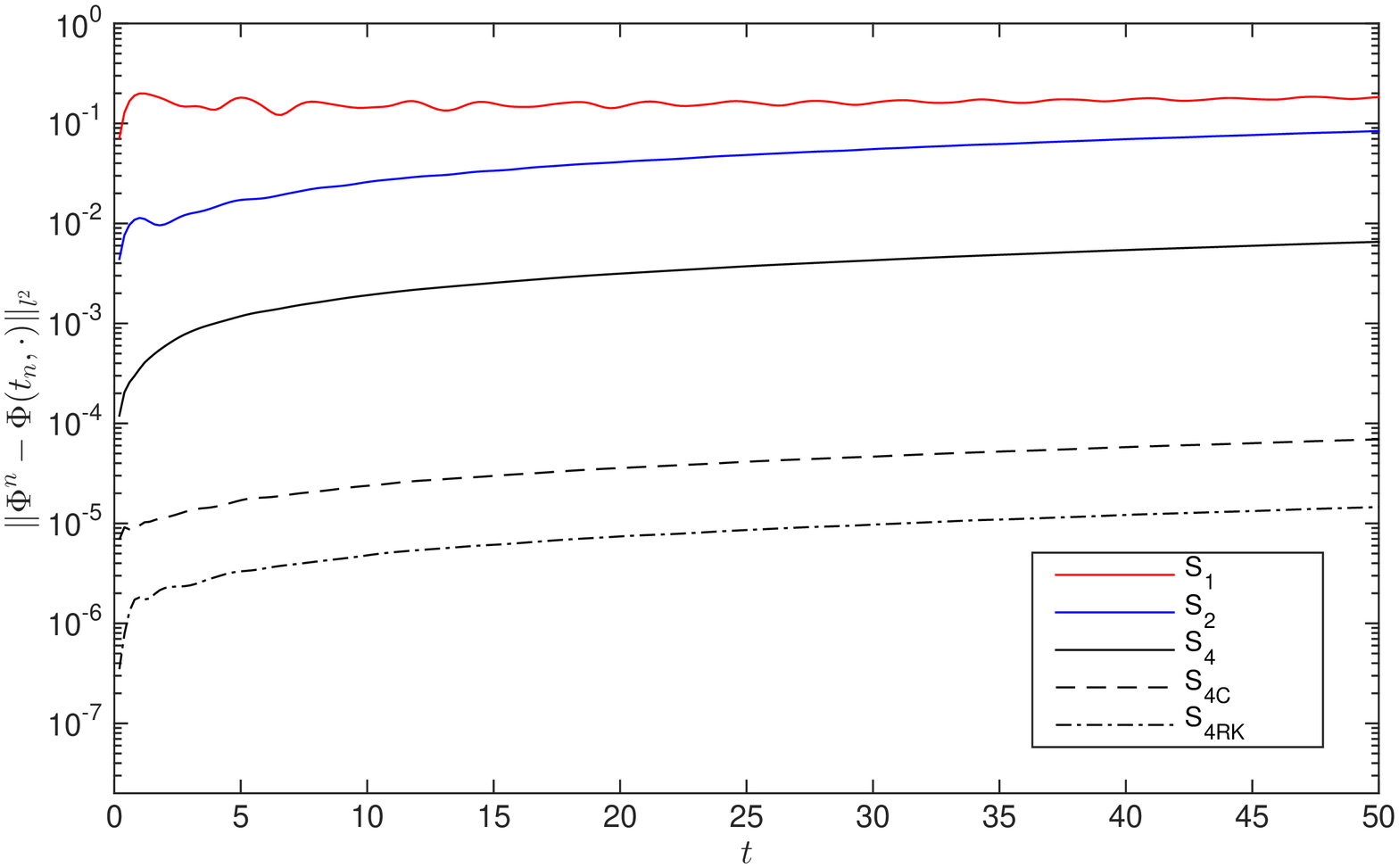,height=7cm,width=12.5cm}}
	\caption{Time evolution of the errors $e_\Phi(t)$ under
$h = \frac{1}{16}$ and $\tau = 0.1$ over long time of
 different time-splitting methods for the  Dirac equation (\ref{LDirac1d2d}) in 1D.}
	\label{longtime}
\end{figure}

From Table \ref{Table 1} and Figure \ref{cmp}, we can draw the following conclusions: (i) The $S_1$ method is first-order in time,
the $S_2$ method is second-order in time, and
the $S_4$, $S_\text{4c}$ and $S_\text{4RK}$ methods are all fourth-order in time (cf. Table \ref{Table 1} and Figure \ref{cmp} left).
(ii) For any fixed mesh $h$ and time $\tau$, the computational time for $S_1$ and $S_2$ are
quite similar, the computational time of $S_\text{4c}$, $S_4$ and $S_\text{4RK}$ are about two times, three times and six times of the $S_2$ method, respectively (cf. Table \ref{Table 1}). (iii) Among the three fourth-order time-splitting methods, $S_\text{4c}$ and $S_\text{4RK}$ are quite similar in terms
of numerical errors for any fixed $\tau$ and they are much smaller than
that of the $S_4$ method, especially when the $\tau$ is not so small (cf. Table \ref{Table 1} and Figure \ref{cmp} left).
(iv) For the constants in front of the convergence rates of different methods, $C_4\gg C_1\sim C_2 \gg \widehat{C_4}\sim \widetilde{C}_4$ (cf.
Figure \ref{cmp} right). (v) For the $S_4$ method, it suffers from convergence rate reduction when the time step is not small and a very large
constant in front of the convergence rate. Thus this method is, in general,
to be avoided in practical computation, which has been observed when it is
applied for the nonlinear Schr\"{o}dinger equation too \cite{Tha}.

To compare the long time behavior of different time-splitting methods,
Figure \ref{longtime} depicts $e_\Phi(t)$ under
mesh size $h = \frac{1}{16}$ and time step $\tau = 0.1$ for
$0\le t\le T:=50$.

From Figure \ref{longtime}, we can observe: (i) The
errors increase very fast when $t$ is small, e.g. $0\le t\le O(1)$,
and they almost don't change when $t\gg1$, thus they are suitable for long time simulation, especially the fourth-order methods. (ii)
When $t$ is not large, the error of the $S_4$ method
is about 10 times bigger than that of the $S_\text{4c}$  method;
however, when $t\gg1$, it becomes about $100$ times larger.
(iii) The error of the $S_\text{4RK}$ method is always the smallest
among all the time-splitting methods.

\bigskip

Based on the efficiency and accuracy as well as long time behavior, in conclusion, for the
three fourth-order time-splitting methods, $S_\text{4c}$ is more accurate than $S_4$ and it is more efficient than $S_\text{4RK}$. Thus the $S_\text{4c}$ method is highly recommended for studying the dynamics of the Dirac equation, especially in 1D.

\subsection{An example in 2D}
For simplicity, here we only compare the three fourth-order integrators,
i.e., $S_\text{4c}$, $S_4$ and $S_\text{4RK}$ via an example in 2D.
In order to do so,
in the Dirac equation (\ref{LDirac1d2d}), we take $d=2$, $\eps = \delta = \nu = 1$ and take the potential in honey-comb form
\be\begin{split}
\label{potential_2D}
V(\mathbf{x}) &= \cos\left(\frac{4\pi}{\sqrt{3}}\mathbf{e_1}\cdot\mathbf{x}\right) +
\cos\left(\frac{4\pi}{\sqrt{3}}\mathbf{e}_2\cdot\mathbf{x}\right) + \cos\left(\frac{4\pi}{\sqrt{3}}\mathbf{e}_3
\cdot\mathbf{x}\right),\\
 A_1(\mathbf{x}) &= A_2(\mathbf{x}) = 0,\qquad  \bx\in{\mathbb R}^2,
\end{split}
\ee
with
\begin{equation}
\mathbf{e}_1 = (-1, 0)^T, \quad \mathbf{e}_2 = (1 / 2, \sqrt{3} / 2)^T, \quad \mathbf{e}_3 = (1 / 2, -\sqrt{3} / 2)^T.
\end{equation}
The initial data in \eqref{initial11} is taken as:
\begin{equation}
\label{initial3}
\phi_1(0, \bx) =  e^{-\frac{x^2 + y^2}{2}}, \quad \phi_2(0, \bx) = e^{-\frac{(x - 1)^2 + y^2}{2}}, \qquad \bx=(x, y)^T\in\mathbb{R}^2.
\end{equation}
The problem is solved numerically on a bounded domain
$\Omega = (-10, 10)\times (-10, 10)$.

Similar to the 1D example,
we obtain a numerical `exact' solution
by using the $S_\text{4c}$ method with a
fine mesh size $h_e = \frac{1}{32}$ and a
small time step $\tau_e = 10^{-4}$.
The error for the numerical solution $\Phi^n$ with mesh size $h$ and time step $\tau$ is quantified as
\begin{equation}
\label{err2D}
e_\Phi(t_n) = {\|\Phi^n - \Phi(t_n, \cdot)\|_{l^2}}
={h\sqrt{\sum_{j = 0}^{M - 1}\sum_{l = 0}^{M - 1}|\Phi(t_n, x_j,y_l) - \Phi_{jl}^n|^2}}.
\end{equation}

Similar to the 1D case, in order to compare the spatial errors, we take time step $\tau=\tau_e =  10^{-4}$
such that the temporal discretization error could be negligible.
Table \ref{spatial_2D-error} lists numerical errors $e_\Phi(t=2)$
for different time-splitting methods under different mesh size $h$.
In order to compare the temporal errors, we take mesh size  $h=h_e = \frac{1}{32}$
such that the spatial discretization error could be negligible.
Table \ref{Table 2} lists numerical errors $e_\Phi(t=2)$
for different time-splitting methods under different time step $\tau$.

\begin{table}[t!]\renewcommand{\arraystretch}{1.2}
	\centering
	\begin{tabular}{|c|cccc|}
		\hline
		&  $h_0=1 / 2$ & $h_0 / 2$ & $h_0 / 2^2$ & $h_0 / 2^3$  \\
		\hline
		$S_4$  & 1.10 & 1.01E-1& 3.83E-4 &7.33E-10\\
		\hline
		$S_\text{4c}$  & 1.10 & 1.01E-1 & 3.83E-4 & 7.33E-10\\
		\hline
		$S_\text{4RK}$  & 1.10 & 1.01E-1 & 3.83E-4 & 7.34E-10\\
		\hline
	\end{tabular}
	\caption{Spatial errors $e_\Phi(t=2)$ of different time-splitting methods under different mesh size $h$ for the  Dirac equation (\ref{LDirac1d2d}) in 2D. }\label{spatial_2D-error}
\end{table}

\begin{table}[t!]\renewcommand{\arraystretch}{1.7}
	\centering
	\small
	\begin{tabular}{|c|c|ccccccc|}
		\hline
		\multicolumn{2}{|c|}{}& $\tau_0 = 1 / 2$ & $\tau_0 / 2$ & $\tau_0 / 2^2$ &
		$\tau_0 / 2^3$ & $\tau_0 / 2^4$ & $\tau_0 / 2^5$  & $\tau_0 / 2^6$ \\
		\hline
		\multirow{3}{*}{{$S_4$}}& Error  & 	4.33E-1 &	2.57E-2 &	3.53E-3 &	2.83E-4 &	1.88E-5 &	1.20E-6 &	7.51E-8 \\
		&Order  & -- &	4.07 &	2.87 &	3.64 &	3.91 &	3.98 &	3.99 \\
		&CPU Time  &  0.20 &	0.26 &	0.45 &	1.04 &	1.63 &	3.37 &	\textbf{6.54} \\
		\hline
		\multirow{3}{*}{{$S_{4c}$}} & Error  & 	6.75E-2 &	3.18E-3 &	7.91E-5 &	4.70E-6 &	2.91E-7 &	1.81E-8 &	1.13E-9 \\
		& Order  & --  &	4.41 &	5.33 &	4.07 &	4.01 &	4.00 &	4.00 \\
		& CPU Time  &  0.12 &	0.28 &	0.31 &	0.55 &	1.11 &	2.09 &	\textbf{4.14} \\
		\hline
		\multirow{3}{*}{{$S_\text{4RK}$}} & Error  &  8.32E-3 &	3.56E-4 &	7.42E-6 &	4.43E-7 &	2.75E-8 &	1.71E-9 &	1.07E-10 \\
		& Order  & -- &	4.55 &	5.59 &	4.07 &	4.01 &	4.00 &	4.00 \\
		& CPU Time  &  	0.26 &	0.43 &	0.87 &	1.52 &	2.92 &	6.20 &	\textbf{11.74} \\
		\hline
	\end{tabular}
\caption{Temporal errors $e_\Phi(t=2)$ of different fourth order time-splitting methods under different time step $\tau$ for the  Dirac equation (\ref{LDirac1d2d}) in 2D. Here we also list convergence rates and computational time (CPU time in seconds) for comparison.}
\label{Table 2}
\end{table}

From Tables \ref{spatial_2D-error}\&\ref{Table 2}, we can draw the following conclusions:
(i) All the three methods are spectrally accurate in space and fourth-order in time. (ii) For any fixed mesh size $h$ and time step $\tau$, the computational times of the $S_4$ and $S_\text{4RK}$ methods are approximately 1.5 times and 3 times more than that of the $S_\text{4c}$ method, respectively.
(iii) $S_\text{4c}$ and $S_\text{4RK}$ are quite similar in terms
of numerical errors for any fixed $\tau$ and the errors are much smaller than
that of the $S_4$ method, especially when $\tau$ is not so small.
(iv) Again, order reduction in time was observed in the $S_4$ method when
$\tau$ is not small, however, there is almost no order reduction
in time for the  $S_\text{4c}$ and $S_\text{4RK}$ methods.

\bigskip

Again, based on the efficiency and accuracy for the Dirac equation in high dimensions, in conclusion, for the
three fourth-order time-splitting methods, $S_\text{4c}$ is more accurate than $S_4$ and it is more efficient than $S_\text{4RK}$. Thus the $S_\text{4c}$ method is highly recommended for studying the dynamics of the Dirac equation in high dimensions, especially without magnetic potential.

\section{Spatial/temporal resolution of the $S_\text{4c}$ method in different parameter regimes} \label{sec6}
\setcounter{equation}{0}
\setcounter{table}{0}
\setcounter{figure}{0}

In this section, we study numerically temproal/spatial resolution
of the fourth-order compact time-splitting Fourier pseudospectral
$S_\text{4c}$ method \eqref{S4c} for the Dirac equation in different parameter
regimes. We take $d=1$ and the electromagnetic potentials as
\eqref{potentials} in Dirac equation (\ref{LDirac1d2d}).
To quantify the numerical error, we adapt the relative errors of the wave
function $\Phi$, the total probability density $\rho$ and the current
${\bf J}$ as
\begin{equation}
e_{\Phi}^r(t_n) = \frac{\|\Phi^n - \Phi(t_n, \cdot)\|_{l^2}}{\|\Phi(t_n, \cdot)\|_{l^2}}, \quad e_{\rho}^r(t_n) = \frac{\|\rho^n - \rho(t_n, \cdot)\|_{l^2}}{\|\rho(t_n, \cdot)\|_{l^2}}, \quad
e_{{\bf J}}^r(t_n) = \frac{\|{\bf J}^n - {\bf J}(t_n, \cdot)\|_{l^2}}{\|{\bf J}(t_n, \cdot)\|_{l^2}},
\end{equation}
where $\rho^n$ and ${\bf J}^n$ are obtained from the wave function $\Phi^n$
via \eqref{density11} and \eqref{current11} with $d=1$, respectively.
Again, the numerical `exact' solution is obtained by using the $S_\text{4c}$
method with a very fine mesh $h=h_e$ and a very small time step $\tau=\tau_e$.

\subsection{In the nonrelativistic limit regime}
Here we take  $\delta = \nu = 1$, $\eps\in(0,1]$ and
the initial data in \eqref{initial11} is taken as \eqref{initial2}.
In this parameter regime, the solution propagates waves
with wavelength at $O(1)$ and $O(\eps^2)$ in space and time,
respectively.
The problem is solved numerically on a bounded domain
$\Omega = (-32, 32)$, i.e. $a=-32$ and $b=32$.
Similar to the second-order time-splitting Fourier pseudospectral
method \cite{BCJT2}, the $S_\text{4c}$ method converges uniformly with respect to
$\eps\in(0,1]$ at spectral order in space. Detailed numerical results
are omitted here for brevity. Here we only present temporal errors
by taking  $h=h_e=\frac{1}{16}$ so that the spatial discretization error
could be negligible.
Table \ref{nonrelativistic_temporal_wave} shows
the temporal errors $e_{\Phi}^r(t=6)$ for the wave function
under different $\tau$ and $\eps\in(0,1]$. Similarly, Tables \ref{nonrelativistic_temporal_density} and
\ref{nonrelativistic_temporal_current} depict the temporal errors $e_{\rho}^r(t=6)$ and $e_{{\bf J}}^r(t=6)$ for the
probability and current, respectively.

\begin{table}[t!]\renewcommand{\arraystretch}{1.2}
	\centering
	\begin{tabular}{ccccccc}
		\hline
		& $\tau_0 = 1$ & $\tau_0 / 2^2$ & $\tau_0 / 2^4$ & $\tau_0 / 2^6$ & $\tau_0 / 2^8$
		& $\tau_0 / 2^{10}$\\
		\hline
		$\eps_0 = 1$ & 2.24E-1 & \textbf{5.07E-4} & 1.95E-6 & 7.63E-9 & $<$1E-10 & $<$1E-10\\
		order & -- &	\textbf{4.39} &	4.01 &	4.00 &	-- &--\\
		\hline
		$\eps_0 / 2$ & 1.18 & 1.05E-2 &	\textbf{3.61E-5} & 1.40E-7 &	5.67E-10 & $<$1E-10\\
		order & -- &	3.41 &	\textbf{4.09} &	4.00 &	3.97 &	--\\
		\hline
		$\eps_0 / 2^2$ & 1.46 &	2.07E-1 & 1.69E-3 &	\textbf{6.09E-6} & 2.37E-8 &	$<$1E-10\\
		order & -- &	1.41 &	3.47 &	\textbf{4.06} &	4.00 &	--\\
		\hline
		$\eps_0 / 2^3$ & 1.41 &	1.50 &	5.88E-2 & 3.84E-4 &	\textbf{1.39E-6} & 5.40E-9\\
		order & -- &	-0.04 &	2.33 &	3.63 &	\textbf{4.06} &	4.00\\
		\hline
		$\eps_0 / 2^4$ & 1.43 &	1.47 &	6.80E-1 &	1.46E-2 &	9.33E-5 &	\textbf{3.38E-7}\\
		order & -- &	-0.02 &	0.56 &	2.77 &	3.65 &	\textbf{4.05}\\
		\hline
	\end{tabular}
\caption{Temporal errors $e_{\Phi}^r(t=6)$ of $S_\text{4c}$ under different $\tau$ and $\eps$ for the  Dirac equation (\ref{LDirac1d2d}) in 1D
	in the nonrelativistic limit regime.}
\label{nonrelativistic_temporal_wave}
\end{table}

\begin{table}[t!]\renewcommand{\arraystretch}{1.2}
	\centering
	\begin{tabular}{ccccccc}
		\hline
		& $\tau_0 = 1$ & $\tau_0 / 2^2$ & $\tau_0 / 2^4$ & $\tau_0 / 2^6$ & $\tau_0 / 2^8$
		& $\tau_0 / 2^{10}$\\
		\hline
		$\eps_0 = 1$ & 1.71E-1 & \textbf{3.73E-4} & 1.44E-6 & 5.62E-9 & $<$1E-10 & $<$1E-10\\
		order & -- &	\textbf{4.42} &	4.01 &	4.00 &	-- &	--\\
		\hline
		$\eps_0 / 2$ & 1.31 & 7.17E-3 &	\textbf{2.45E-5} & 9.50E-8 &	3.94E-10 & $<$1E-10\\
		order & -- &	3.76 &	\textbf{4.10} &	4.01 &	3.96 &	--\\
		\hline
		$\eps_0 / 2^2$ & 8.19E-1 &	2.20E-1 & 8.16E-4 &	\textbf{2.92E-6} & 1.13E-8 &	$<$1E-10\\
		order & -- &	0.95 &	4.04 &	\textbf{4.06} &	4.00 &	--\\
		\hline
		$\eps_0 / 2^3$ & 8.75E-1 & 4.77E-1 & 5.76E-2 & 1.65E-4 & \textbf{5.89E-7} & 2.29E-9\\
		order & -- &	0.44 &	1.52 &	4.22 &	\textbf{4.07} &	4.00\\
		\hline
		$\eps_0 / 2^4$ & 1.00 &	1.12 &	2.04E-1 & 1.49E-2 &	4.03E-5 & \textbf{1.43E-7}\\
		order & -- &	-0.08 &	1.23 &	1.88 &	4.27 &	\textbf{4.07}\\
		\hline
	\end{tabular}
	\caption{Temporal errors $e_{\rho}^r(t=6)$ of $S_\text{4c}$ under different $\tau$ and $\eps$ for the  Dirac equation (\ref{LDirac1d2d}) in 1D
		in the nonrelativistic limit regime.}
	\label{nonrelativistic_temporal_density}
\end{table}

\begin{table}[t!]\renewcommand{\arraystretch}{1.2}
	\centering
	\begin{tabular}{ccccccc}
		\hline
		& $\tau_0 = 1$ & $\tau_0 / 2^2$ & $\tau_0 / 2^4$ & $\tau_0 / 2^6$ & $\tau_0 / 2^8$
		& $\tau_0 / 2^{10}$\\
		\hline
		$\eps_0 = 1$ & 2.92E-1 & \textbf{6.76E-4} &	2.61E-6 & 1.02E-8 &	$<$1E-10 & $<$1E-10\\
		order & -- &	\textbf{4.38} &	4.01 &	4.00 &	-- &	--\\
		\hline
		$\eps_0 / 2$ & 1.30 & 1.98E-2 &	\textbf{6.88E-5} & 2.67E-7 & 1.06E-9 & $<$1E-10\\
		order & -- &	3.02 & \textbf{4.09} &	4.00 &	3.99 &	--\\
		\hline
		$\eps_0 / 2^2$ & 1.29 &	2.98E-1 & 3.40E-3 &	\textbf{1.23E-5} & 4.76E-8 & $<$1E-10\\
		order & -- &	1.06 &	3.23 &	\textbf{4.06} &	4.00 &	--\\
		\hline
		$\eps_0 / 2^3$ & 1.21 &	1.29 &	8.82E-2 & 7.85E-4 &	\textbf{2.85E-6} & 1.11E-8\\
		order & -- &	-0.05 &	1.94 &	3.41 &	\textbf{4.05} &	4.00\\
		\hline
		$\eps_0 / 2^4$ & 1.52 &	1.44 &	1.30 &	2.41E-2	& 1.92E-4 &	\textbf{6.98E-7}\\
		order & -- &	0.04 &	0.07 &	2.88 &	3.48 &	\textbf{4.05}\\
		\hline
	\end{tabular}
	\caption{Temporal errors $e_{\bf J}^r(t=6)$ of $S_\text{4c}$ under different $\tau$ and $\eps$ for the  Dirac equation (\ref{LDirac1d2d}) in 1D
		in the nonrelativistic limit regime.}
	\label{nonrelativistic_temporal_current}
\end{table}

From Tables \ref{nonrelativistic_temporal_wave}-\ref{nonrelativistic_temporal_current},
when $\tau \lesssim \eps^2$, fourth-order convergence is observed
for the $S_\text{4c}$ method in the relative error for the wave function,
probability and current. This suggests that the $\eps$-scalability for
the $S_\text{4c}$ method in the nonrelativistic limit regime is:
$h=O(1)$ and  $\tau = O(\eps^2)$. In addition, noticing
$\Phi=O(1)$, $\rho=O(1)$ and ${\bf J}=O(\eps^{-1})$ when $0\le \eps\ll1$,
we can formally observe the following error bounds for $0<\eps\le 1$, $\tau\lesssim \eps^2$ and
$0\le n \le \frac{T}{\tau}$
\be\begin{split}
&\|\Phi^n - \Phi(t_n, \cdot)\|_{l^2}\lesssim h^{m_0}+\frac{\tau^4}{\eps^6},\quad
\|\rho^n - \rho(t_n, \cdot)\|_{l^2}\lesssim h^{m_0}+\frac{\tau^4}{\eps^6},\\
&\|{\bf J}^n - {\bf J}(t_n, \cdot)\|_{l^2}\lesssim \frac{1}{\eps}\left( h^{m_0}+\frac{\tau^4}{\eps^6}\right).
\end{split}
\ee
where $m_0\ge2$ depends on the regularity of the solution. Rigorous mathematical
justification is still on-going.

\subsection{In the semiclassical limit regime}
Here we take $\varepsilon = \nu = 1$, $\delta \in (0, 1]$.
The initial data in \eqref{initial11} is taken as
\be\begin{split}
\phi_1(0, x) &= \frac{1}{2}e^{-4x^2}e^{iS_0(x) / \delta}\left(1 + \sqrt{1+S_0'(x)^2 }\right), \\
\phi_2(0, x) &= \frac{1}{2}e^{-4x^2}e^{iS_0(x) / \delta}S_0'(x), \quad x\in{\mathbb R},
\end{split}
\ee
with
\begin{equation}
S_0(x) = \frac{1}{40}\big(1 + \cos(2\pi x)\big), \qquad x\in{\mathbb R}.
\end{equation}
In this parameter regime, the solution propagates waves
with wavelength at $O(\delta)$ in both space and time.
The problem is solved numerically on a bounded domain
$\Omega = (-16, 16)$, i.e. $a=-16$ and $b=16$.

Table \ref{semiclassical_spatial_wave} shows
the spatial errors $e_{\Phi}^r(t=2)$ for the wave function
under different $h$ and $\delta\in(0,1]$ with $\tau=\tau_e=10^{-4}$
such that the temporal discretization error could be negligible.
Tables \ref{semiclassical_spatial_density} and
\ref{semiclassical_spatial_current} depict the spatial errors $e_{\rho}^r(t=2)$ and $e_{{\bf J}}^r(t=2)$ for the
probability and current, respectively.
Similarly, Table \ref{semiclassical_temporal_wave} shows
the temporal errors $e_{\Phi}^r(t=2)$ for the wave function
under different $\tau$ and $\delta\in(0,1]$ with $h=
h_e=\frac{1}{128}$ so that the spatial discretization error
could be negligible.
 Tables \ref{semiclassical_temporal_density} and
\ref{semiclassical_temporal_current} depict the temporal errors $e_{\rho}^r(t=2)$ and $e_{{\bf J}}^r(t=2)$ for the
probability and current, respectively.

\begin{table}[t!]\renewcommand{\arraystretch}{1.2}
	\centering
	\begin{tabular}{cccccccc}
		\hline
		& $h_0 = 1$ & $h_0 / 2$ & $h_0 / 2^2$ & $h_0 / 2^3$ & $h_0 / 2^4$ & $h_0 / 2^5$ & $h_0 / 2^6$\\
		\hline
		$\delta_0 = 1$ & \textbf{8.25E-1} &	2.00E-1 & 9.52E-3 &	6.66E-6 & 3.78E-10 & $<$1E-10 & $<$1E-10\\
		\hline
		$\delta_0 / 2$ & 1.20 &	\textbf{7.40E-1} & 5.31E-2 &	8.87E-5 & 3.43E-10 & $<$1E-10 & $<$1E-10\\
		\hline
		$\delta_0 / 2^2$ & 1.41 & 9.89E-1 &	\textbf{5.12E-1} & 3.81E-3 &	9.24E-10 & $<$1E-10 & $<$1E-10\\
		\hline
		$\delta_0 / 2^3$ & 1.76 & 1.21 & 7.30E-1 &	\textbf{2.76E-1} & 1.91E-5 &	4.17E-10 & $<$1E-10\\
		\hline
		$\delta_0 / 2^4$ & 1.37 & 1.36 & 1.36 &	5.31E-1 & \textbf{1.54E-1} &	5.31E-10 & $<$1E-10\\
		\hline
		$\delta_0 / 2^5$ & 2.44 & 1.92 & 1.36 &	1.36 & 4.36E-1 & \textbf{5.49E-2} &	2.90E-10\\
		\hline
	\end{tabular}
	\caption{Spatial errors $e_{\Phi}^r(t=2)$ of $S_\text{4c}$ under different $h$ and $\delta$ for the  Dirac equation (\ref{LDirac1d2d}) in 1D in the semiclassical limit regime.}
	\label{semiclassical_spatial_wave}
\end{table}

\begin{table}[t!]\renewcommand{\arraystretch}{1.2}
	\centering
	\begin{tabular}{cccccccc}
		\hline
		& $h_0 = 1$ & $h_0 / 2$ & $h_0 / 2^2$ & $h_0 / 2^3$ & $h_0 / 2^4$ & $h_0 / 2^5$ & $h_0 / 2^6$\\
		\hline
		$\delta_0 = 1$ & \textbf{5.83E-1} & 1.39E-1 & 8.27E-3 & 4.36E-6 & 4.92E-10 & $<$1E-10 & $<$1E-10\\
		\hline
		$\delta_0 / 2$ & 1.29 &	\textbf{5.22E-1} & 3.71E-2 &	5.56E-5 & 2.79E-10 & $<$1E-10 &	$<$1E-10\\
		\hline
		$\delta_0 / 2^2$ & 9.22E-1 & 7.44E-1 &	\textbf{2.41E-1} & 1.54E-3 &	6.75E-10 &	$<$1E-10 & $<$1E-10\\
		\hline
		$\delta_0 / 2^3$ & 1.63 & 9.39E-1 &	6.11E-1 & \textbf{6.33E-2} &	4.78E-6 & 8.19E-10 & $<$1E-10\\
		\hline
		$\delta_0 / 2^4$ & 2.04 & 1.40 & 1.00 &	3.57E-1 & \textbf{1.97E-2} &	6.76E-10 & $<$1E-10\\
		\hline
		$\delta_0 / 2^5$ & 5.81 & 3.65 & 1.07 &	1.01 &	1.86E-1 & \textbf{3.35E-3} &	5.67E-10\\
		\hline
	\end{tabular}
	\caption{Spatial errors $e_{\rho}^r(t=2)$ of $S_\text{4c}$ under different $h$ and $\delta$ for the  Dirac equation (\ref{LDirac1d2d}) in 1D in the semiclassical limit regime.}
	\label{semiclassical_spatial_density}
\end{table}

\begin{table}[t!]\renewcommand{\arraystretch}{1.2}
	\centering
	\begin{tabular}{cccccccc}
		\hline
		& $h_0 = 1$ & $h_0 / 2$ & $h_0 / 2^2$ & $h_0 / 2^3$ & $h_0 / 2^4$ & $h_0 / 2^5$ & $h_0 / 2^6$\\
		\hline
		$\delta_0 = 1$ & \textbf{8.07E-1} & 1.67E-1 & 1.05E-2 &	5.69E-6 & 5.10E-10 & $<$1E-10 &	$<$1E-10\\
		\hline
		$\delta_0 / 2$ & 1.45 &	\textbf{6.89E-1} & 4.28E-2 &	6.46E-5 & 3.06E-10 & $<$1E-10 &	$<$1E-10\\
		\hline
		$\delta_0 / 2^2$ & 1.94 & 1.05 & \textbf{3.52E-1} &	2.13E-3 & 7.96E-10 & $<$1E-10 &	$<$1E-10\\
		\hline
		$\delta_0 / 2^3$ & 2.52 & 1.03 & 7.07E-1 &	\textbf{1.24E-1} & 7.75E-6 &	8.16E-10 &	$<$1E-10\\
		\hline
		$\delta_0 / 2^4$ & 2.85 & 1.77 & 1.10 &	5.84E-1 & \textbf{4.72E-2} &	6.75E-10 &	$<$1E-10\\
		\hline
		$\delta_0 / 2^5$ & 3.88 & 4.06 & 1.11 &	1.07 &	3.81E-1 & \textbf{1.22E-2} &	5.63E-10\\
		\hline
	\end{tabular}
	\caption{Spatial errors $e_{{\bf J}}^r(t=2)$ of $S_\text{4c}$ under different $h$ and $\delta$ for the  Dirac equation (\ref{LDirac1d2d}) in 1D in the semiclassical limit regime.}
	\label{semiclassical_spatial_current}
\end{table}

\begin{table}[t!]\renewcommand{\arraystretch}{1.2}
	\centering
	\begin{tabular}{cccccccc}
		\hline
		& $\tau_0 = 1$ & $\tau_0 / 2$ & $\tau_0 / 2^2$ & $\tau_0 / 2^3$ & $\tau_0 / 2^4$ & $\tau_0 / 2^5$ &
		$\tau_0 / 2^6$ \\
		\hline
		$\delta_0 = 1$ & 1.60E-1 & \textbf{1.58E-2} & 5.09E-4 &	2.08E-5 & 1.27E-6 &	7.89E-8 & 4.94E-9 \\
		order & -- & \textbf{3.34} &	4.96 &	4.61 &	4.04 &	4.01 &	4.00 \\
		\hline
		$\delta_0 / 2$ & 8.66E-1 &	1.48E-1 & \textbf{7.17E-3} &	3.90E-4 & 2.41E-5 &	1.50E-6 & 9.39E-8\\
		order & -- &	2.55 & \textbf{4.36} & 4.20 & 4.02 &	4.00 & 4.00 \\
		\hline
		$\delta_0 / 2^2$ & 1.26 & 9.52E-1 &	1.38E-1 & \textbf{7.38E-3} &	4.50E-4 & 2.80E-5 &	1.75E-6 \\
		order & -- &	0.40 & 2.78 &\textbf{ 4.23} & 4.03 &	4.01 &	4.00\\
		\hline
		$\delta_0 / 2^3$ & 1.45 & 1.20 & 9.94E-1 & 1.62E-1 & \textbf{9.11E-3} & 5.57E-4 & 3.46E-5 \\
		order & -- &	0.27 &	0.27 &	2.62 &	\textbf{4.15} &	4.03 &	4.01\\
		\hline
		$\delta_0 / 2^4$ & 1.40 & 1.44 & 1.12 &	9.46E-1 & 2.62E-1 &	\textbf{1.50E-2} & 9.15E-4 \\
		order & -- &	-0.04 &	0.36 &	0.25 & 1.85 & \textbf{4.13} & 4.03 \\
		\hline
		$\delta_0 / 2^5$ & 1.44 & 1.44 & 1.42 &	1.22 & 1.07 & 4.43E-1 &	\textbf{2.83E-2}\\
		order & -- &	-0.01 &	0.03 &	0.22 &	0.19 &	1.27 &	\textbf{3.97}\\
		\hline
	\end{tabular}
	\caption{Temporal errors $e_{\Phi}^r(t=2)$ of $S_\text{4c}$ under different $\tau$ and $\delta$ for the  Dirac equation (\ref{LDirac1d2d}) in 1D in the semiclassical limit regime.}
	\label{semiclassical_temporal_wave}
\end{table}

\begin{table}[t!]\renewcommand{\arraystretch}{1.2}
	\centering
	\begin{tabular}{cccccccc}
		\hline
		& $\tau_0 = 1$ & $\tau_0 / 2$ & $\tau_0 / 2^2$ & $\tau_0 / 2^3$ & $\tau_0 / 2^4$ & $\tau_0 / 2^5$ &
		$\tau_0 / 2^6$ \\
		\hline
		$\delta_0 = 1$ & 1.15E-1 & \textbf{1.23E-2} & 4.11E-4 & 1.70E-5 & 1.03E-6 & 6.40E-8 & 4.11E-9 \\
		order & -- &	\textbf{3.23} & 4.90 & 4.59 & 4.05 &	4.01 &	3.96 \\
		\hline
		$\delta_0 / 2$ & 5.05E-1 &	9.20E-2 & \textbf{4.93E-3} &	2.36E-4 & 1.44E-5 &	8.98E-7 & 5.62E-8\\
		order & -- &	2.45 & \textbf{4.22} & 4.39 & 4.03 &	4.01 &	4.00 \\
		\hline
		$\delta_0 / 2^2$ & 7.69E-1 & 4.22E-1 & 4.32E-2 & \textbf{2.85E-3} & 1.73E-4 & 1.08E-5 &	6.72E-7\\
		order & -- &	0.86 & 3.29 & \textbf{3.92} & 4.04 &	4.01 & 4.00 \\
		\hline
		$\delta_0 / 2^3$ & 1.28 & 9.03E-1 &	5.67E-1 & 3.77E-2 &	\textbf{2.03E-3} & 1.23E-4 &	7.66E-6 \\
		order & -- &	0.51 &	0.67 &	3.91 &	\textbf{4.21} &	4.04 &	4.01 \\
		\hline
		$\delta_0 / 2^4$ & 8.80E-1 & 1.25 &	9.86E-1 & 7.53E-1 &	2.58E-2 & \textbf{1.35E-3} &	8.15E-5\\
		order & -- &	-0.50 &	0.34 & 0.39 & 4.87 & \textbf{4.26} &	4.05 \\
		\hline
		$\delta_0 / 2^5$ & 9.60E-1 & 9.90E-1 & 1.09 & 1.08 & 8.82E-1 & 2.59E-2 & \textbf{1.16E-3}\\
		order & -- &	-0.04 &	-0.14 &	0.02 &	0.29 &	5.09 &	\textbf{4.48} \\
		\hline
	\end{tabular}
	\caption{Temporal errors $e_{\rho}^r(t=2)$ of $S_\text{4c}$ under different $\tau$ and $\delta$ for the  Dirac equation (\ref{LDirac1d2d}) in 1D in the semiclassical limit regime.}
	\label{semiclassical_temporal_density}
\end{table}

\begin{table}[t!]\renewcommand{\arraystretch}{1.2}
	\centering
	\begin{tabular}{cccccccc}
		\hline
		& $\tau_0 = 1$ & $\tau_0 / 2$ & $\tau_0 / 2^2$ & $\tau_0 / 2^3$ & $\tau_0 / 2^4$ & $\tau_0 / 2^5$ &
		$\tau_0 / 2^6$ \\
		\hline
		$\delta_0 = 1$ & 1.98E-1 & \textbf{2.21E-2} & 6.42E-4 & 2.34E-5 & 1.42E-6 & 8.84E-8 & 5.55E-9 \\
		order & -- &	\textbf{3.16} &	5.11 &	4.78 &	4.04 &	4.01 &	3.99\\
		\hline
		$\delta_0 / 2$ & 6.61E-1 &	1.93E-1 & \textbf{8.72E-3} & 4.34E-4 & 2.67E-5 &	1.66E-6 & 1.04E-7\\
		order & -- &	1.78 &	\textbf{4.47} & 4.33 & 4.02 & 4.01 &	4.00 \\
		\hline
		$\delta_0 / 2^2$ & 1.25 & 6.66E-1 &	1.46E-1 & \textbf{8.44E-3} &	5.16E-4 & 3.21E-5 &	2.00E-6 \\
		order & -- &	0.91 &	2.19 &	\textbf{4.12} &	4.03 &	4.01 &	4.00 \\
		\hline
		$\delta_0 / 2^3$ & 1.57 & 1.19 & 7.29E-1 & 1.23E-1 & \textbf{7.10E-3} & 4.35E-4 & 2.71E-5\\
		order & -- &	0.39 &	0.71 &	2.57 &	\textbf{4.11} &	4.03 &	4.01\\
		\hline
		$\delta_0 / 2^4$ & 1.04 & 1.47 & 1.15 &	8.24E-1 & 9.50E-2 &	\textbf{5.86E-3} & 3.60E-4 \\
		order & -- &	-0.50 &	0.35 &	0.48 &	3.12 &	\textbf{4.02} &	4.02\\
		\hline
		$\delta_0 / 2^5$ & 1.02 & 1.14 & 1.19 &	1.19 &	9.39E-1 & 7.34E-2 &	\textbf{5.22E-3}\\
		order & -- &	-0.16 &	-0.06 &	0.01 &	0.34 &	3.68 &	\textbf{3.81}\\
		\hline
	\end{tabular}
	\caption{Temporal errors $e_{\bf J}^r(t=2)$ of $S_\text{4c}$ under different $\tau$ and $\delta$ for the  Dirac equation (\ref{LDirac1d2d}) in 1D in the semiclassical limit regime.}
	\label{semiclassical_temporal_current}
\end{table}

From Tables \ref{semiclassical_spatial_wave}-\ref{semiclassical_spatial_current},
when $h \lesssim \delta$, spectral convergence (in space) is observed
for the $S_\text{4c}$ method in the relative error for the wave function,
probability and current. Similarly, from Tables \ref{semiclassical_temporal_wave}-\ref{semiclassical_temporal_current},
when $\tau \lesssim \delta$, fourth-order convergence (in time) is observed
for the $S_\text{4c}$ method  in the relative error for the wave function,
probability and current.
 These suggest that the $\delta$-scalability for
the $S_\text{4c}$ method in the semiclassical limit regime is:
$h=O(\delta)$ and  $\tau = O(\delta)$. In addition, noticing
$\Phi=O(1)$, $\rho=O(1)$ and ${\bf J}=O(1)$ when $0\le \delta\ll1$,
we can formally observe the following error bounds for $0<\delta\le 1$, $\tau\lesssim \delta$, $h\lesssim \delta$ and
$0\le n \le \frac{T}{\tau}$
\be\begin{split}
&\|\Phi^n - \Phi(t_n, \cdot)\|_{l^2}\lesssim \frac{h^{m_0}}{\delta^{m_0}}+\frac{\tau^4}{\delta^4},\quad
\|\rho^n - \rho(t_n, \cdot)\|_{l^2}\lesssim \frac{h^{m_0}}{\delta^{m_0}}+\frac{\tau^4}{\delta^4},\\
&\|{\bf J}^n - {\bf J}(t_n, \cdot)\|_{l^2}\lesssim  \frac{h^{m_0}}{\delta^{m_0}}+\frac{\tau^4}{\delta^4}.
\end{split}
\ee
where $m_0\ge2$ depends on the regularity of the solution. Rigorous mathematical
justification is still on-going.

\subsection{In the simultaneously nonrelativistic and massless limit regime}
We take $d=1$, $\delta=1$ and $\nu=\eps$ in (\ref{LDirac1d2d}) with
$\eps\in(0,1]$. The initial data in \eqref{initial11} is taken as \eqref{initial2}.
In this parameter regime, the solution propagates waves with wavelength at $O(1)$ and $O(\eps)$
in space and time, respectively.
The problem is solved numerically on a bounded domain
$\Omega = (-128, 128)$, i.e. $a=-128$ and $b=128$ by $S_\text{4c}$.
Similar to the nonrelativistic limit regime, the $S_\text{4c}$ method converges uniformly with respect to
$\eps\in(0,1]$ at spectral order in space. Detailed numerical results
are omitted here for brevity. Here we only present temporal errors
by taking  $h=h_e=\frac{1}{16}$ so that the spatial discretization error
could be negligible.
Table \ref{nrelative_mless_temporal_wave} shows
the temporal errors $e_{\Phi}^r(t=2)$ for the wave function
under different $\tau$ and $\eps\in(0,1]$. Similarly, Tables
\ref{nrelative_mless_temporal_density} and
\ref{nrelative_mless_temporal_current}
 depict the temporal errors $e_{\rho}^r(t=2)$
 and $e_{{\bf J}}^r(t=2)$ for the
probability and current, respectively.

\begin{table}[t!]\renewcommand{\arraystretch}{1.2}
	\centering
	\begin{tabular}{cccccccc}
		\hline
		& $\tau_0 = 1$ & $\tau_0 / 2$ & $\tau_0 / 2^2$ & $\tau_0 / 2^3$ & $\tau_0 / 2^4$ & $\tau_0 / 2^5$ &
		$\tau_0 / 2^6$\\
		\hline
		$\eps_0 = 1$ & \textbf{1.12E-1} & 4.20E-3 &	2.18E-4 & 1.33E-5 &	8.30E-7 &	5.18E-8 &	3.24E-9 \\
		order & -- &	4.74 &	4.27 &	4.03 &	4.01 &	4.00 &	4.00\\
		\hline
		$\eps_0 / 2$ & 4.72E-1 & \textbf{3.66E-2} &	1.17E-3 &	6.64E-5 &	4.09E-6 &	2.55E-7 &	1.59E-8 \\
		order & -- &	\textbf{3.69} &	4.97 &	4.14 &	4.02 &	4.01 &	4.00 \\
		\hline
		$\eps_0 / 2^2$ & 1.14 &	2.72E-1 &	\textbf{1.27E-2} &	3.64E-4 &	2.10E-5 &	1.30E-6 &	8.08E-8 \\
		order & -- &	2.07 &	\textbf{4.42} &	5.12 &	4.11 &	4.02 &	4.00 \\
		\hline
		$\eps_0 / 2^3$ & 1.29 &	5.84E-1 &	1.60E-1 &	\textbf{5.19E-3} &	1.41E-4 &	8.22E-6 &	5.07E-7 \\
		order & -- &	1.14 &	1.87 &	\textbf{4.94} &	5.20 &	4.10 &	4.02\\
		\hline
		$\eps_0 / 2^4$ & 1.40 &	7.31E-1 &	3.40E-1 &	9.81E-2 &	\textbf{2.46E-3} &	6.16E-5 &	3.58E-6 \\
		order & -- &	0.94 &	1.10 &	1.79 &	\textbf{5.32} &	5.32 &	4.10\\
		\hline
		$\eps_0 / 2^5$ & 1.39 &	1.06 &	3.90E-1 &	2.09E-1 &	6.32E-2 &	\textbf{1.27E-3} &	2.84E-5\\
		order & -- &	0.40 &	1.44 &	0.90 &	1.72 &	\textbf{5.64} &	5.48\\
		\hline
		$\eps_0 / 2^6$ & 1.48 &	1.48 &	5.90E-1 &	2.19E-1 &	1.32E-1 &	4.21E-2 &	\textbf{7.04E-4} \\
		order & -- &	0.00 &	1.32 &	1.43 &	0.72 &	1.65 &	\textbf{5.90}\\
		\hline
	\end{tabular}
	\caption{Temporal errors $e_{\Phi}^r(t=2)$ of $S_\text{4c}$ under different $\tau$ and $\eps$ for the  Dirac equation (\ref{LDirac1d2d}) in 1D in the simultaneously nonrelativistic and massless limit regime.}
	\label{nrelative_mless_temporal_wave}
\end{table}

\begin{table}[t!]\renewcommand{\arraystretch}{1.2}
	\centering
	\begin{tabular}{cccccccc}
		\hline
		& $\tau_0 = 1$ & $\tau_0 / 2$ & $\tau_0 / 2^2$ & $\tau_0 / 2^3$ & $\tau_0 / 2^4$ & $\tau_0 / 2^5$ &
		$\tau_0 / 2^6$ \\
		\hline
		$\eps_0 = 1$ & \textbf{8.62E-2} &	3.48E-3 &	1.91E-4 &	1.17E-5 &	7.28E-7 &	4.54E-8 &	2.82E-9 \\
		order & -- &	4.63 &	4.19 &	4.03 &	4.01 &	4.00 &	4.01\\
		\hline
		$\eps_0 / 2$ & 3.56E-1 &	\textbf{2.97E-2} &	7.90E-4 &	4.56E-5 &	2.82E-6 &	1.76E-7 &	1.10E-8 0\\
		order & -- &	\textbf{3.59} &	5.23 &	4.12 &	4.01 &	4.00 &	4.00\\
		\hline
		$\eps_0 / 2^2$ & 9.98E-1 &	2.83E-1 &	\textbf{1.22E-2} &	2.54E-4 &	1.45E-5 &	8.95E-7 &	5.57E-8\\
		order & -- &	1.82 &	\textbf{4.53} &	5.59 &	4.13 &	4.02 &	4.01 \\
		\hline
		$\eps_0 / 2^3$ & 8.15E-1 &	5.58E-1 &	1.60E-1 &	\textbf{4.18E-3} &	9.00E-5 &	5.29E-6 &	3.27E-7\\
		order & -- &	0.55 &	1.80 &	\textbf{5.26} &	5.54 &	4.09 &	4.02\\
		\hline
		$\eps_0 / 2^4$ & 9.32E-1 &	7.05E-1 &	3.32E-1 &	1.02E-1 &	\textbf{1.69E-3} &	3.69E-5 &	2.19E-6 \\
		order & -- &	0.40 &	1.09 &	1.70 &	\textbf{5.92} &	5.52 &	4.08 \\
		\hline
		$\eps_0 / 2^5$ & 1.05 &	6.88E-1 &	3.28E-1 &	2.07E-1 &	6.70E-2 &	\textbf{8.68E-4} &	1.63E-5 \\
		order & -- &	0.61 &	1.07 &	0.67 &	1.63 &	\textbf{6.27} &	5.73 \\
		\hline
		$\eps_0 / 2^6$ & 8.39E-1 &	8.04E-1 &	4.76E-1 &	1.72E-1 &	1.27E-1 &	4.33E-2 &	\textbf{5.49E-4}\\
		order & -- &	0.06 &	0.76 &	1.47 &	0.44 &	1.55 &	\textbf{6.30}\\
		\hline
	\end{tabular}
	\caption{Temporal errors $e_{\rho}^r(t=2)$ of $S_\text{4c}$ under different $\tau$ and $\eps$ for the  Dirac equation (\ref{LDirac1d2d}) in 1D in the simultaneously nonrelativistic and massless limit regime.}
	\label{nrelative_mless_temporal_density}
\end{table}

\begin{table}[t!]\renewcommand{\arraystretch}{1.2}
	\centering
	\begin{tabular}{cccccccc}
		\hline
		& $\tau_0 = 1$ & $\tau_0 / 2$ & $\tau_0 / 2^2$ & $\tau_0 / 2^3$ & $\tau_0 / 2^4$ & $\tau_0 / 2^5$ &
		$\tau_0 / 2^6$\\
		\hline
		$\eps_0 = 1$ & \textbf{2.03E-1} &	7.11E-3 &	4.03E-4 &	2.47E-5 &	1.54E-6 &	9.61E-8 &	5.98E-9 \\
		order & -- &	4.84 &	4.14 &	4.03 &	4.01 &	4.00 &	4.01\\
		\hline
		$\eps_0 / 2$ & 7.37E-1 & \textbf{5.58E-2} &	1.89E-3 &	1.11E-4 &	6.84E-6 &	4.26E-7 &	2.66E-8 \\
		order & -- &	\textbf{3.72} &	4.88 &	4.09 &	4.02 &	4.00 &	4.00\\
		\hline
		$\eps_0 / 2^2$ & 1.34 & 4.30E-1 &	\textbf{1.81E-2} &	5.59E-4 &	3.31E-5 &	2.05E-6 &	1.28E-7 \\
		order & -- &	1.64 &	\textbf{4.57} &	5.01 &	4.08 &	4.02 &	4.00 \\
		\hline
		$\eps_0 / 2^3$ & 1.20 &	7.03E-1 &	2.30E-1 &	\textbf{6.14E-3} &	1.89E-4 &	1.13E-5 &	7.00E-7 \\
		order & -- &	0.77 &	1.61 &	\textbf{5.23} &	5.02 &	4.06 &	4.01 \\
		\hline
		$\eps_0 / 2^4$ & 1.36 &	1.04 &	4.15E-1 &	1.31E-1 &	\textbf{2.52E-3} &	7.59E-5 &	4.57E-6 \\
		order & -- &	0.39 &	1.32 &	1.66 &	\textbf{5.71} &	5.05 &	4.05 \\
		\hline
		$\eps_0 / 2^5$ & 1.63 &	1.32 &	5.79E-1 &	2.47E-1 &	8.28E-2 &	\textbf{1.27E-3} & 3.26E-5 \\
		order & -- &	0.30 &	1.19 &	1.23 &	1.58 &	\textbf{6.03} &	5.28 \\
		\hline
		$\eps_0 / 2^6$ & 1.38 &	1.47 &	8.97E-1	& 3.04E-1 &	1.52E-1 &	5.54E-2 &	\textbf{7.52E-4}\\
		order & -- &	-0.09 &	0.71 &	1.56 &	1.00 &	1.45 &	\textbf{6.20} \\
		\hline
	\end{tabular}
	\caption{Temporal errors $e_{\bf J}^r(t=2)$ of $S_\text{4c}$ under different $\tau$ and $\eps$ for the  Dirac equation (\ref{LDirac1d2d}) in 1D  in the simultaneously nonrelativistic and massless limit regime.}
	\label{nrelative_mless_temporal_current}
\end{table}

From Tables \ref{nrelative_mless_temporal_wave}-\ref{nrelative_mless_temporal_current},
when $\tau \lesssim \eps$, fourth-order convergence is observed
for the $S_\text{4c}$ method in the relative error for the wave function,
probability and current. This suggests that the $\eps$-scalability for
the $S_\text{4c}$ method in the simultaneously nonrelativistic and massless limit regime is:
$h=O(1)$ and  $\tau = O(\eps)$. In addition, noticing
$\Phi=O(1)$, $\rho=O(1)$ and ${\bf J}=O(\eps^{-1})$ when $0\le \eps\ll1$,
we can formally observe the following error bounds for $0<\eps\le 1$, $\tau\lesssim \eps$ and
$0\le n \le \frac{T}{\tau}$
\be\begin{split}
&\|\Phi^n - \Phi(t_n, \cdot)\|_{l^2}\lesssim h^{m_0}+\frac{\tau^4}{\eps^3},\quad
\|\rho^n - \rho(t_n, \cdot)\|_{l^2}\lesssim h^{m_0}+\frac{\tau^4}{\eps^3},\\
&\|{\bf J}^n - {\bf J}(t_n, \cdot)\|_{l^2}\lesssim \frac{1}{\eps}\left( h^{m_0}+\frac{\tau^4}{\eps^3}\right).
\end{split}
\ee
where $m_0\ge2$ depends on the regularity of the solution. Rigorous mathematical
justification is still on-going.

\bigskip
Based on the discussion in Section 1 and numerical comparison results
in this section, Table \ref{Properties} lists spatial/temporal
wavelengths of the Dirac equation under different parameter regimes and the corresponding spatial/temporal resolution
of the $S_\text{4c}$ method.

\begin{table}[t!]\renewcommand{\arraystretch}{1.2}
	\small
	\centering
	\begin{tabular}{c|c|c|c|c|c|c}
		\hline
		& \tabincell{c}{Spatial\\wavelength} & \tabincell{c}{Temporal\\ wavelength} & \tabincell{c}{Spatial\\ accuracy} &
		\tabincell{c}{Temporal\\accuracy} & \tabincell{c}{Spatial\\ resolution} & \tabincell{c}{Temporal \\resolution}\\
		\hline
		\tabincell{c}{Standard\\regime}& $O(1)$ & $O(1)$ & spectral & $O(\tau^4)$ & $O(1)$ & $O(1)$\\
		\hline
		\tabincell{c}{Nonrelativistic\\limit regime}& $O(1)$ & $O(\eps^2)$ & spectral & $O(\frac{\tau^4}{\eps^6})$ & $O(1)$ & $O(\eps^2)$\\
		\hline
		\tabincell{c}{Semiclassical\\limit regime}& $O(\delta)$ & $O(\delta)$ & spectral & $O(\frac{\tau^4}{\delta^4})$ &
		$O(\delta)$ & $O(\delta)$\\
		\hline
		\tabincell{c}{Nonrelativistic \\ \&massless\\limit regime} &$O(1)$ & $O(\eps)$ & spectral & $O(\frac{\tau^4}{\eps^3})$ & $O(1)$ & $O(\eps)$\\
		\hline
		\tabincell{c}{Massless\\limit regime} &$O(1)$ & $O(1)$ & spectral & $O(\tau^4)$ & $O(1)$ & $O(1)$\\
		\hline
	\end{tabular}
    \caption{Spatial/temporal
wavelengths of the Dirac equation under different parameter regimes and the corresponding spatial/temporal resolution
of the $S_\text{4c}$ method.}
    \label{Properties}
\end{table}

\section{Conclusion} \label{sec7}

A new fourth-order compact time-splitting Fourier
pseudospectral ($S_\text{4c}$) method was proposed for the Dirac equation.
It is explicit, fourth-order in time and  spectral accurate
in space. One major advantage in the method is to avoid using
negative time steps in integrating sub-problems via the double commutator. Numerical results showed that it is much more accurate than first-order and second-order
time-splitting methods, and it is more accurate
than the standard fourth-order time-splitting method and is more efficient
than the partitioned Runge-Kutta time-splitting method, especially in 1D or in
high dimensions without magnetic potentials.
In addition, it is very robust for simulating long time dynamics.
Spatial and temporal resolution of the proposed numerical method
was studied numerically for the Dirac equation under different parameter
regimes including the nonrelativistic limit regime, the semiclassical
limit regime, and the simultaneously nonrelativistic and massless limit regime. Based on our extensive numerical results, for  numerical simulation of the dynamics of the Dirac equation in 1D or in high dimensions without magnetic potential, the $S_\text{4c}$ method  is a very efficient and accurate as well as simple numerical method.
Of course, for the Dirac equation in high dimensions with magnetic potential,
$S_\text{4RK}$ is a good choice.

\bibliographystyle{model1-num-names}
\bibliography{<your-bib-database>}

\vspace{10pt}
\setcounter{equation}{0}  


\begin{center}
	{\bf Appendix A}. Proof of Lemma 3.3 on double commutator of the Dirac equation in 2D
\end{center}
\setcounter{equation}{0}
\renewcommand{\theequation}{A.\arabic{equation}}

\begin{proof}
Combining \eqref{TW2d1} and \eqref{dcm2}, we obtain
\begin{equation}
\label{splitA1}
[W, [T, W]] = -\frac{1}{\eps}[W, [\sigma_1\partial_1, W]] -\frac{1}{\eps} [W, [\sigma_2\partial_2, W]]
- \frac{i\nu}{\delta\eps^2}[W, [\sigma_3, W]].
\end{equation}
From \eqref{Pauli}, we have
\be\begin{split}
\label{Pauliper}
&\sigma_j^2=I_2, \quad \sigma_j\sigma_l=-\sigma_l\sigma_j, \qquad 1\le j\ne l\le 3,\\
&\sigma_1\sigma_2=i\sigma_3, \quad \sigma_2\sigma_3=i\sigma_1,
\quad \sigma_3\sigma_1=i\sigma_2.
\end{split}
\ee
Noticing \eqref{TW2d1}, \eqref{dcm1} and \eqref{Pauliper}, we get
\begin{eqnarray}
\label{2Dcommutator_1}
&&[W, [\sigma_1\partial_1, W]]\nn\\
&&=  -\frac{1}{\delta^2}\Big(2\big(V(\bx)I_2 - A_1(\bx)\sigma_1 -
A_2(\bx)\sigma_2\big)(\sigma_1\partial_1)
\big(V(\bx)I_2 - A_1(\bx)\sigma_1 - A_2(\bx)\sigma_2\big) \nn\\
&&\ \ \ - \big(V(\bx)I_2 - A_1(\bx)\sigma_1 - A_2(\bx)\sigma_2\big)^2(\sigma_1\partial_1) -
(\sigma_1\partial_1)\big(V(\bx)I_2 - A_1(\bx)\sigma_1 - A_2(\bx)\sigma_2\big)^2\Big)\nn\\
&&=   -\frac{2}{\delta^2}\sigma_1A_2(\bx)\sigma_2\big(\partial_1V(\bx)I_2 - \partial_1A_1(\bx)\sigma_1 -
\partial_1A_2(\bx)\sigma_2\big)\nn\\
&&\ \ \  - \frac{2}{\delta^2}\sigma_1\big(V(\bx)I_2 - A_1(\bx)\sigma_1 + A_2(\bx)\sigma_2\big)
\big(V(\bx)I_2 - A_1(\bx)\sigma_1 -
A_2(\bx)\sigma_2\big)\partial_1\nn\\
&&\ \ \  + \frac{1}{\delta^2}\sigma_1\big(V(\bx)I_2 - A_1(\bx)\sigma_1 +
A_2(\bx)\sigma_2\big)^2\partial_1 + \frac{1}{\delta^2}\sigma_1\big(V(\bx)I_2 - A_1(\bx)\sigma_1 -
A_2(\bx)\sigma_2\big)^2\partial_1\nn\\
& &\ \ \ - \frac{2}{\delta^2}\sigma_1A_2(\bx)\sigma_2
\big(\partial_1V(\bx)I_2 - \partial_1A_1(\bx)\sigma_1 -
\partial_1A_2(\bx)\sigma_2\big)\nn\\
 &&=  -\frac{4}{\delta^2}A_2(\bx)\big(\partial_1V(\bx)\sigma_1\sigma_2 + \partial_1A_1(\bx)\sigma_2 -
 \partial_1A_2(\bx)\sigma_1\big) + \frac{4}{\delta^2}A_2^2(\bx)\sigma_1\partial_1\nn\\
&&\ \ \  - \frac{4}{\delta^2}A_1(\bx)A_2(\bx)\sigma_2\partial_1\nn\\
&&= \frac{4}{\delta^2}\big(A_2^2(\bx)\sigma_1 - A_1(\bx)A_2(\bx)\sigma_2\big)\partial_1
+ \frac{4}{\delta^2}A_2(\bx)\big(\partial_1A_2(\bx)\sigma_1 - \partial_1A_1(\bx)\sigma_2\big)\nn\\
&&\ \ \ - \frac{4i}{\delta^2}A_2(\bx)\partial_1V(\bx)\sigma_3.
\end{eqnarray}
\begin{eqnarray}
\label{2Dcommutator_3}
[W, [\sigma_3, W]] &= & -\frac{1}{\delta^2}\Big(2\big(V(\bx)I_2 - A_1(\bx)\sigma_1 - A_2(\bx)\sigma_2\big)
\sigma_3\big(V(\bx)I_2 - A_1(\bx)\sigma_1 - A_2(\bx)\sigma_2\big)\nn\\
& &- \big(V(\bx)I_2 - A_1(\bx)\sigma_1 - A_2(\bx)\sigma_2\big)^2\sigma_3
- \sigma_3\big(V(\bx)I_2 - A_1(\bx)\sigma_1 - A_2(\bx)\sigma_2\big)^2\Big)\nn\\
&= & \frac{2}{\delta^2}\sigma_3\big(V(\bx)I_2 + A_1(\bx)\sigma_1 + A_2(\bx)\sigma_2\big)
\big(A_1(\bx)\sigma_1 + A_2(\bx)\sigma_2\big) \nn\\
& &- \frac{2}{\delta^2}\sigma_3\big(A_1(\bx)\sigma_1 + A_2(\bx)\sigma_2\big)
\big(V(\bx)I_2 - A_1(\bx)\sigma_1 - A_2(\bx)\sigma_2\big)\nn\\
&= & \frac{4}{\delta^2}\big(A_1^2(\bx) + A_2^2(\bx)\big)\sigma_3.
\end{eqnarray}
\begin{eqnarray}
\label{2Dcommutator_2}
[W, [\sigma_2\partial_2, W]] &= &- \frac{4}{\delta^2}\big(A_1(\bx)A_2(\bx)\sigma_1
- A_1^2(\bx)\sigma_2\big)\partial_2 -\frac{4}{\delta^2}A_1(\bx)\big(\partial_2A_2(\bx)\sigma_1 - \partial_2A_1(\bx)\sigma_2\big)\nn\\
&&+ \frac{4i}{\delta^2}A_1(\bx)\partial_2V(\bx)\sigma_3 .
\end{eqnarray}
Plugging (\ref{2Dcommutator_1}), (\ref{2Dcommutator_2})
and (\ref{2Dcommutator_3}) into (\ref{splitA1}), after some computation,
we can get \eqref{commutator_2D}.
\end{proof}

\begin{center}
	{\bf Appendix B}. Proof of Lemma 3.4 on double commutator of the Dirac equation in 3D
\end{center}
\setcounter{equation}{0}
\renewcommand{\theequation}{B.\arabic{equation}}

\begin{proof}
Combining \eqref{TW3d1} and \eqref{dcm2}, we obtain
\begin{equation}
\label{split2}
[W, [T, W]] = -\frac{1}{\eps}[W, [\alpha_1\partial_1, W]] -\frac{1}{\eps} [W, [\alpha_2\partial_2, W]]
-\frac{1}{\eps} [W, [\alpha_3\partial_3, W]] - \frac{i\nu}{\delta\eps^2} [W, [\beta, W]].
\end{equation}
From \eqref{matrices} and \eqref{gammam}, we have
\be\begin{split}
\label{matricper}
&\beta^2=I_4, \quad \alpha_j^2=I_4, \quad \alpha_j\alpha_l=-\alpha_l\alpha_j,\\
&\beta\alpha_j = -\alpha_j\beta,
\quad \gamma\alpha_j=\alpha_j\gamma,  \qquad 1\le j\ne l\le 3,\\
&\alpha_1\alpha_2=i\gamma\alpha_3, \quad \alpha_2\alpha_3=i\gamma\alpha_1, \quad \alpha_3\alpha_1=i\gamma\alpha_2.
\end{split}
\ee
Noticing \eqref{TW3d1}, \eqref{dcm1} and \eqref{matricper}, we get
\begin{eqnarray}
\label{3Dcommutator_4}
[W, [\beta, W]] &= & -\frac{1}{\delta^2}\bigg(2\Big(V(\bx)I_4 - \sum_{j = 1}^3A_j(\bx)\alpha_j\Big)\beta
\Big(V(\bx)I_4 - \sum_{j = 1}^3A_j(\bx)\alpha_j\Big)\nn\\
&& - \Big(V(\bx)I_4 - \sum_{j = 1}^3A_j(\bx)\alpha_j\Big)^2\beta
- \beta\Big(V(\bx)I_4 - \sum_{j = 1}^3A_j(\bx)\alpha_j\Big)^2\bigg)\nn\\
&= & -\frac{2}{\delta^2}\beta\Big(V(\bx)I_4 + \sum_{j = 1}^3A_j(\bx)\alpha_j\Big)
\Big(V(\bx)I_4 - \sum_{j = 1}^3A_j(\bx)\alpha_j\Big)\nn\\
&& + \frac{1}{\delta^2}\beta\Big(V(\bx)I_4 + \sum_{j = 1}^3A_j(\bx)\alpha_j\Big)^2 +
\frac{1}{\delta^2}\beta\Big(V(\bx)I_4 - \sum_{j = 1}^3A_j(\bx)\alpha_j\Big)^2\nn\\
&= & \frac{4}{\delta^2}\big(A_1^2(\bx) + A_2^2(\bx) + A_3^2(\bx)\big)\beta.
\end{eqnarray}
\bea\label{3Dcommutator_1}
&&[W, [\alpha_1\partial_1, W]]\nn\\
 &&=  -\frac{1}{\delta^2}\bigg(2\Big(V(\bx)I_4
- \sum_{j = 1}^3A_j(\bx)\alpha_j\Big)(\alpha_1\partial_1)\Big(V(\bx)I_4
- \sum_{j = 1}^3A_j(\bx)\alpha_j\Big)\nn\\
 &&\ \ \ - \Big(V(\bx)I_4 - \sum_{j = 1}^3A_j(\bx)\alpha_j\Big)^2(\alpha_1\partial_1)
 - (\alpha_1\partial_1)\Big(V(\bx)I_4 - \sum_{j = 1}^3A_j(\bx)\alpha_j\Big)^2\bigg)\nn\\
 &&=  -\frac{4}{\delta^2}\alpha_1\big(A_2(\bx)\alpha_2 + A_3(\bx)\alpha_3\big)\big(\partial_1V(\bx)I_4 -
 \partial_1A_1(\bx)\alpha_1 - \partial_1A_2(\bx)\alpha_2 - \partial_1A_3(\bx)\alpha_3\big)\nn\\
 &&\ \  \ + \frac{1}{\delta^2}\alpha_1\bigg(\Big(V(\bx)I_4 - A_1(\bx)\alpha_1 + A_2(\bx)\alpha_2
 + A_3(\bx)\alpha_3\Big)^2
 + \Big(V(\bx)I_4 - \sum_{j = 1}^3A_j(\bx)\alpha_j\Big)^2\nn\\
 &&\ \ \  - 2\Big(V(\bx)I_4 - A_1(\bx)\alpha_1 +
 A_2(\bx)\alpha_2 + A_3(\bx)\alpha_3\Big)\Big(V(\bx)I_4 - \sum_{j = 1}^3A_j(\bx)\alpha_j\Big)\bigg)\partial_1,\nn\\
 &&=  \frac{4}{\delta^2}\big(A_2(\bx)\alpha_2 + A_3(\bx)\alpha_3\big)\alpha_1\big(\partial_1V(\bx)I_4 -
 \partial_1A_1(\bx)\alpha_1 - \partial_1A_2(\bx)\alpha_2 - \partial_1A_3(\bx)\alpha_3\big)\nn\\
 &&\ \ \ + \frac{4}{\delta^2}\Big(\big(A_2^2(\bx) + A_3^2(\bx)\big)\alpha_1 - A_1(\bx)A_2(\bx)\alpha_2 -
 A_1(\bx)A_3(\bx)\alpha_3\Big)\partial_1\nn\\
 &&= \frac{4}{\delta^2}\Big(\big(A_2(\bx)\partial_1A_2(\bx) + A_3(\bx)\partial_1A_3(\bx)\big)\alpha_1
 - A_2(\bx)\partial_1A_1(\bx)\alpha_2 -  A_3(\bx)\partial_1A_1(\bx)\alpha_3\nn\\
 &&\ \ \ + \big(iA_2(\bx)\partial_1A_3(\bx) - iA_3(\bx)\partial_1A_2(\bx)\big)\gamma
 + iA_3(\bx)\partial_1V(\bx)\gamma\alpha_2 - iA_2(\bx)\partial_1V(\bx)\gamma\alpha_3\Big)\nn\\
 &&\ \ \ + \frac{4}{\delta^2}\Big(\big(A_2^2(\bx) + A_3^2(\bx)\big)\alpha_1 - A_1(\bx)A_2(\bx)\alpha_2 -
 A_1(\bx)A_3(\bx)\alpha_3\Big)\partial_1.
\end{eqnarray}
\begin{eqnarray}
\label{3Dcommutator_2}
&&[W, [\alpha_2\partial_2, W]]\nn\\
&&= \frac{4}{\delta^2}\Big(-A_1(\bx)\partial_2A_2(\bx)\alpha_1
+ \big(A_1(\bx)\partial_2A_1(\bx) + A_3(\bx)\partial_2A_3(\bx)\big)\alpha_2 - A_3(\bx)\partial_2A_2(\bx)\alpha_3\nn\\
&&\ \ \  + \big(iA_3(\bx)\partial_2A_1(\bx) - iA_1(\bx)\partial_2A_3(\bx)\big)\gamma
- iA_3(\bx)\partial_2V(\bx)\gamma\alpha_1 + iA_1(\bx)\partial_2V(\bx)\gamma\alpha_3\Big)\nn\\
&&\ \ \  + \frac{4}{\delta^2}\Big(\big(A_1^2(\bx) + A_3^2(\bx)\big)\alpha_2 - A_2(\bx)A_1(\bx)\alpha_1 -
A_2(\bx)A_3(\bx)\alpha_3\Big)\partial_2.
\end{eqnarray}
\begin{eqnarray}
\label{3Dcommutator_3}
&&[W, [\alpha_3\partial_3, W]]\nn\\
 &&=  \frac{4}{\delta^2}\Big(- A_1(\bx)\partial_3A_3(\bx)\alpha_1
- A_2(\bx)\partial_3A_3(\bx)\alpha_2 + \big(A_1(\bx)\partial_3A_1(\bx) + A_2(\bx)\partial_3A_2(\bx)\big)\alpha_3\nn\\
& &\ \ \ + \big(iA_1(\bx)\partial_3A_2(\bx) - iA_2(\bx)\partial_3A_1(\bx)\big)\gamma +
iA_2(\bx)\partial_3V(\bx)\gamma\alpha_1 - iA_1(\bx)\partial_3V(\bx)\gamma\alpha_2\Big)\nn\\
&&\ \ \ + \frac{4}{\delta^2}\Big(\big(A_1^2(\bx) + A_2^2(\bx)\big)\alpha_3 - A_3(\bx)A_1(\bx)\alpha_1 -
A_3(\bx)A_2(\bx)\alpha_2\Big)\partial_3.
\end{eqnarray}
Plugging (\ref{3Dcommutator_1}), (\ref{3Dcommutator_2}), (\ref{3Dcommutator_3})
and (\ref{3Dcommutator_4}) into (\ref{split2}), after some computation,
we obtain \eqref{commutator_3D}.
 \end{proof}

\end{document}